\documentclass[a4paper,11pt,twoside]{article}

\usepackage{graphicx}
\usepackage{caption}
\usepackage{subcaption}
\usepackage{multirow}
\usepackage[utf8]{inputenc}
\usepackage[T1]{fontenc}
\usepackage{ dsfont }
\usepackage[english]{babel}

\usepackage{charter}

\usepackage{indentfirst}
\usepackage{xcolor}
\usepackage{verbatim}

\usepackage{hyperref}

\usepackage{pifont}

\usepackage[top=3.cm, bottom=4.0cm, left=2.2cm, right=2.2cm]{geometry}
\usepackage{amsmath}
\usepackage{amsthm}	
\usepackage{amsfonts}	
\usepackage{amssymb	
            ,bbm
            ,units 
            ,stmaryrd
           }
\usepackage{enumerate}
\usepackage{ textcomp, gensymb }
\usepackage[square,numbers,sort&compress]{natbib}

\usepackage{
            ulem		
           ,soul		
} 

\numberwithin{equation}{section}
\numberwithin{figure}{section}
 \usepackage[nodayofweek]{datetime}

\newcommand{\real}{{\textrm{Re}}}
\newcommand{\imag}{{\textrm{Im}}}

\title{An iterative method for Helmholtz boundary value problems arising in wave propagation}

\author{
	\normalsize Francisco Bernal\textit{${^{a,}}{{}^*}$} \\
	\small franciscomanuel.bernal@uc3m.es
	\and
	\normalsize Xingyuan Chen\textit{$^{b}$} \\
	\small   X.Chen-176@sms.ed.ac.uk 
	\and
	\normalsize Gon\c calo dos Reis\textit{$^{b,c,}$}\footnote{F.B. acknowledges support from the Madrid Regional Government through project 2018-T1/TIC-10914 and by the Spanish AEI grant PID2020-115088RB-I00. G.d.R. acknowledges support from the \emph{Funda{\c c}$\tilde{\text{a}}$o para a Ci$\hat{e}$ncia e a Tecnologia} (Portuguese Foundation for Science and Technology) through the project UIDB/00297/2020 and UIDP/00297/2020 (Center for Mathematics and Applications, CMA/FCT/UNL).} \\
	\small  G.dosReis@ed.ac.uk
}

\date{%
	\footnotesize 
	$^{a}$~Departament of Mathematics, Carlos III University of Madrid, Spain
	\\
	$^{b}$~School of Mathematics, University of Edinburgh, The King's Buildings, Edinburgh, UK
	\\
	$^{c}$~Centro de Matem\'atica e Aplica\c c$\tilde{\text{o}}$es (CMA), FCT, UNL, Portugal
	\\
	\longdate \today \ (\currenttime)
	\vspace{-1.0cm}
}

\theoremstyle{plain}
\newtheorem{theorem}{Theorem}[section]

\newtheorem{remark}[theorem]{Remark}

\newcommand{\bE}{\mathbb{E}}

\newcommand{\bN}{\mathbb{N}}
\newcommand{\bP}{\mathbb{P}}

\newcommand{\bR}{\mathbb{R}}

\newcommand{\EE}{\mathbb{E}}

\definecolor{darkgreen}{rgb}{0,0.35,0}

\newcommand{\1}{\mathbbm{1}}

\newcommand{\dd}{\mathrm{d}}

\newcommand{\bn}{ \mathbf{n} }
 
\begin{document}

\selectlanguage{english}

\maketitle

\begin{abstract}
The complex Helmholtz equation $(\Delta + k^2)u=f$ (where $k\in{\mathbb R},u(\cdot),f(\cdot)\in{\mathbb C}$) is a mainstay of computational wave simulation. Despite its apparent simplicity, efficient numerical methods are challenging to design and, in some applications, regarded as an open problem. Two sources of difficulty are the large number of degrees of freedom and the indefiniteness of the matrices arising after discretisation. Seeking to meet them within the novel framework of probabilistic domain decomposition, we set out to rewrite the Helmholtz equation into a form amenable to the Feynman-Kac formula for elliptic boundary value problems. We consider two typical scenarios, the scattering of a plane wave and the propagation inside a cavity, and recast them as a sequence of Poisson equations. By means of stochastic arguments, we find a sufficient and simulatable condition for the convergence of the iterations. 
Upon discretisation a necessary condition for convergence can be derived by adding up the iterates using the harmonic series for the matrix inverse---we illustrate the procedure in the case of finite differences. 

From a practical point of view, our results are ultimately of limited scope. Nonetheless, this unexpected --- even paradoxical --- new direction of attack on the Helmholtz equation proposed by this work offers a fresh perspective on this classical and difficult problem. Our results show that there indeed exists a predictable range $k<k_{max}$ in which this new ansatz works with $k_{max}$ being far below the challenging situation. 
\end{abstract}


\small
\tableofcontents
\normalsize

\section{Introduction}\label{S:Intro}
\paragraph{Motivation.} The Helmholtz partial differential equation (PDE) is ubiquitous in acoustics, quantum mechanics, seismology, electromagnetism, and in any field concerned with the theoretical study or numerical simulation of waves  \cite{Courant&Hilbert}. It is a complex-valued, elliptic PDE which governs the Fourier transform of the solution of the wave function in a bounded or unbounded domain $\Omega\subset {\mathbb R}^D$ 
in dimension $D=1,2,3$, and is supplemented with the appropriate boundary conditions (BCs) on $\partial\Omega$ or at infinity
\footnote{See the lecture notes by O.~Runborg, {\em www.csc.kth.se/utbildning/kth/kurser/DN2255/ndiff12/Lecture5.pdf}, for an excellent introductory text.}. We shall  consider boundary value problems (BVPs) of the form
\begin{align}
\label{eq:H0}	
	\Delta u({\bf x}) + k^2 u({\bf x})= f({\bf x})\ \textrm{ if } \ {\bf x}\in\Omega\subset{\mathbb R}^D
 \quad + \quad 
 \textrm{ linear BCs on }\partial\Omega,
\end{align}
where $\Delta$ is the Laplacian operator and $k\in{\mathbb R}$, but $f({\bf x})$ and $u({\bf x})$ are complex valued functions. Intuitively, $f$ is a source that injects time-harmonic waves of angular frequency $\omega$ into $\Omega$, and the solution $u$ is the amplitude of the wave field $v(t,{\bf x})=u({\bf x})e^{-i\omega t}$ (with $i$ standing for the imaginary unit). The wave number $k$ is proportional to the frequency as $k=\omega/c$, where $c$ is the speed of the wave in the medium (also assumed constant). Although non-constant and non-linear coefficients are usual in applications, for the purpose of this article, it suffices to consider \eqref{eq:H0}. For simplicity, we set $c=1$ in the remainder of the article so that $k=\omega$.

Notwithstanding its simplicity, \eqref{eq:H0} can be numerically challenging---even intractable \cite{Zienkiewicz_Unsolved}. To see why, let us focus on two situations: the {\em cavity problem} and the {\em scattering problem}. When $\Omega$ is a bounded simply connected domain (a cavity), \eqref{eq:H0} is well-posed as long as $k^2$ is not an eigenvalue of the Laplacian in $\Omega$ with the same BCs -- the associated frequency is a resonant frequency of $\Omega$. The numerical simulation is delicate when $k^2$ is close to resonance, and this issue is compounded at high frequencies (and in three dimensions) since the density of resonant frequencies to $k$  
grows with $k^{D-1}$ \cite{Carleman}. Resonances disappear if a damping term $+ipu$ (for $p>0$) is included on the left-hand side of \eqref{eq:H0}. Damping is a common physical occurrence that attenuates the amplitude as the wave travels. Including artificial damping is the basis of some numerical methods since the underdamped solution is the viscosity solution of the damped one (i.e.~as $p\to 0^+$ -- see footnote).
\medskip

This paper is mostly motivated by the problem sketched on Figure \ref{F:Scattering} (left), where $f\equiv 0$ and a plane wave $v_{in}(t,{\bf x})=A_{in}e^{-ikx_1}e^{-i\omega t},~{\bf x}=(x_1,x_2)$, $A_{in}\in{\mathbb R}$, emitted  from infinity is scattered by an obstacle $\Omega_{sc}$; without loss of generality, we set $D=2$, and the $x_1$ axis is chosen as the incidence direction. 
\begin{figure}[hbt]
	\centering
\includegraphics[width=0.8\textwidth]{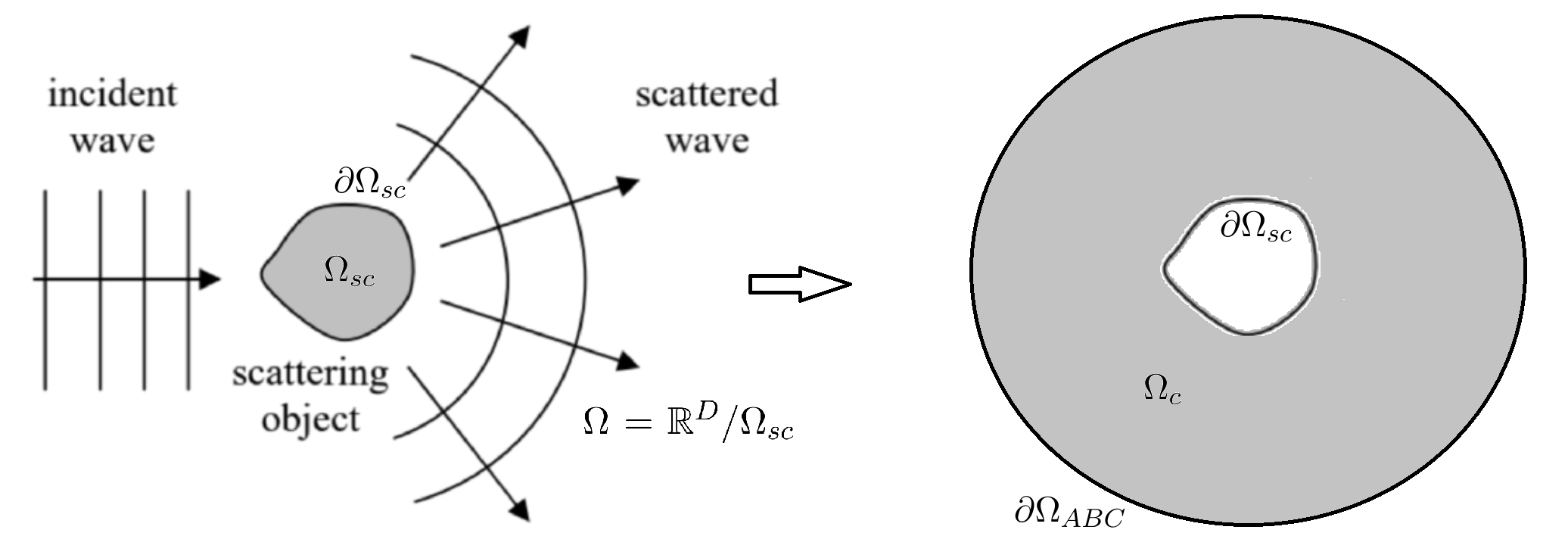}
\caption{Scattering of a plane monochromatic wave (left). The complex Helmholtz BVP, which governs the scattered amplitude, is defined on the exterior of the scatterer $\Omega_{sc}$, but numerically solved in the annular domain $\Omega_c$ between $\partial\Omega_{sc}$ and the computational boundary $\partial\Omega_{ABC}$ (right).}
\label{F:Scattering}
\end{figure}

The total wave field $v(t,{\bf x})$ is the superposition of the incoming and the scattered wave, i.e. $v(t,{\bf x})=u({\bf x})e^{-i\omega t}= v_{in}(t,{\bf x})+u_{sc}({\bf x})e^{-i\omega t}$. Since both $A_{in}e^{-ikx_1}$ (by inspection) and $u({\bf x})$ obey \eqref{eq:H0}, so does the amplitude of the scattered component to the wave field, $u_{sc}({\bf x})$, in the unbounded domain $\Omega={\mathbb R}^D/\Omega_{sc}$ \cite{Harari_BC}. The BCs for $u({\bf x})$ on the scatterer's surface $\partial\Omega_{sc}$ are homogeneous and---depending on the physical situation---they can be of either Dirichlet or Neumann type. 
Consequently, on $\partial\Omega_{sc}$, either $u_{sc}=-A_{in}e^{-ikx_1}$, or $\partial u_{sc}/\partial {\bf n}= -\partial(A_{in}e^{-ikx_1})/\partial {\bf n}$ (where $\partial/\partial {\bf n}$ denotes the normal derivative). 
This annular problem for $u_{sc}$ is well-posed as long as the Sommerfeld radiation condition is upheld and which reads in ${\mathbb R}^2$ as 
\begin{align}
\label{eq:Sommerfeld}	
	\lim\limits_{||{\bf x}||\to\infty}||{\bf x}||^{1/2}\left(\frac{\partial u_{sc}}{\partial r}({\bf x})-iku_{sc}({\bf x})\right)=0,
\end{align}
(where $||{\bf x}||$ and $\partial/\partial r$ are the radial component and derivative, respectively). The condition's meaning is to filter out the (unphysical) scattered waves coming in from infinity.

In order to construct a numerical solution for $u_{sc}({\bf x})$, the scattering  problem is posed on a finite computational domain $\Omega_c\subset \Omega$, such that the artificial computational boundary $\partial\Omega_{ABC}$ lies far out from the scatterer and thus \eqref{eq:H0} is solved in practice on the annular region between $\partial\Omega_{sc}$ and $\partial\Omega_{ABC}$. 
Absorbing BCs (ABCs) are then required on $\partial\Omega_{ABC}$, which are geared to prevent impinging waves from reflecting back into the computational domain \cite{Givoli_BCs}. Assuming that the computational boundary $\partial\Omega_{ABC}$ is a ball centred at the scatterer, and distant enough that the scattered waves can be locally approximated by a plane wave on each point of $\partial\Omega_{ABC}$, a simple yet accurate ABC is the homogeneous Robin BC
\begin{align}
\nonumber
\frac{\partial u_{sc}({\bf x})}{\partial r}-iku_{sc}({\bf x})=0,	
\end{align}
which can be thought of as the BC \eqref{eq:Sommerfeld} enforced on $\partial\Omega_c$ \cite{Harari_BC}.

Regardless of the artificial BCs, the  discretisation of $\Omega_c$ must ensure a sufficient density of nodes throughout so that the numerical approximation can capture the oscillations of the solution. To attain this, the linear density of nodes (in any direction) must be inversely proportional to the solution wavelength, i.e. proportional to $k$. Since the numerical error of \eqref{eq:H0} is dominated by dispersive---rather than dissipative---effects, truncation errors are accumulated over the wavelengths of the numerical solution, causing them to be $k$ times larger than for diffusion-dominated PDEs---this is the notorious pollution error \cite{pollution}. (For instance, with a standard second-order finite differences scheme, the error goes as $k^3$ instead of $k^2$.)  
High-order discretisations, such as spectral methods, are more resilient or even immune to this effect in practice \cite{Meshless}\cite{waveguide}, but rely on wider stencils and so lead to significantly larger matrix bandwidth. In sum, the tradeoff between accuracy, on one side, and size and sparsity of the discretisation matrix, on the other, is worse for wave problems than it is for diffusive ones, and this handicap grows with the frequency / wavenumber.   

The huge algebraic systems arising from realistic simulations \cite{smith1996parallel} must, by default, be iteratively solved on a parallel computer by means of domain decomposition. 
With Helmholtz BVPs for wave problems, the joint requirements for both a large computational domain and adequate sampling lead to notoriously high degrees of freedom (DoFs), and this number (say $N\gg 1$) grows superlinearly with $k$. Assuming that the iterative algebraic method converges in $N$ cycles (as many as there are DoFs), and that the complexity per cycle is proportional to the number of nonzero matrix elements, the cost goes as $N\times N^{D}=N^{1+D}$. Yet achieving convergence in ${\cal O}(N)$ iterations takes efficient preconditioning, which is far from given with Helmholtz's equation since the non-definiteness of the discretisations to \eqref{eq:H0} means that standard, well-proven preconditioners often fail, especially as $k$ grows \cite{Gander} -- this is presently an active area of research  \cite{Erlangga}. 
Summing up, accurately solving Helmholtz equations for wave propagation---especially the scattering of high-frequency waves---may overwhelm the capabilities of the state-of-the-art in distributed computing and domain decomposition for BVPs.

\paragraph{The case for probabilistic domain decomposition (PDD).} PDD is an innovative domain decomposition algorithm for large elliptic BVPs \cite{Acebron_PDD}. The novelty consists in computing the pointwise solution of the BVP on the nodes along subdomain interfaces right away by means of the Feynman-Kac formula \cite{Freidlin85,Milstein&Tretyakov04}. Once the solution along the subdomain interfaces is available, each of the subdomain-restricted problems can be solved independently from one another. This is done embarrassingly in parallel, and involving orders of magnitude fewer DoFs per solve---those in each subdomain only.   
Since the Feynman-Kac representation of the BVP leads naturally to a Monte Carlo method, the contributions from the individual trajectories can be computed in an embarrassingly parallel way, too. The point of PDD is to enhance scalability of domain decomposition with many DoFs and subdomains.  

Superficially, it would seem as though wave propagation problems made an ideal test bed for PDD, since the numerical issues of the former match the theoretical strengths of the latter. Nonetheless, two serious obstacles stand in the way.     

PDD relies on the probabilistic representation of the BVP at hand \cite{Freidlin85}. For example, for the BVP
\begin{align}
\nonumber
\Delta u + c({\bf x})u= 0\textrm{ if ${\bf x}\in\Omega$},\qquad u({\bf x})=0 \textrm{ if ${\bf x}\in\partial\Omega$} 
\end{align}
with $\Omega\subset{\mathbb R}^D$, the Feynman-Kac representation of the solution $u$ at ${\bf x}_0\in\Omega$ is 
\begin{align}
	\label{eq:FK_ex}
u({\bf x}_0)={\mathbb E}\left[\exp\left(\int_0^{\tau^{\Omega}({\bf x}_0)}c({\bf x}_0+\sqrt{2}{\bf W}_s)\,\dd s\right)\right],
\end{align}
as long as the expectation is finite. 
Above, ${\bf W}_t$ is a $D-$dimensional standard Wiener process, and $\tau^{\Omega}({\bf x}_0)$ is the first exit time of ${\bf x}_0+\sqrt{2}({\bf W}_t)_{t\geq 0}$ from $\Omega$. 
If $c({\bf x})\leq 0$ in $\Omega$ and ${\mathbb E}[\tau^{\Omega}({\bf x}_0)]<\infty$, then \eqref{eq:FK_ex} can be shown to hold --- in fact, the second assumption is redundant if $c({\bf x})\leq -c_0^2<0$ for some $c_0$. 
On the other hand, the distribution in \eqref{eq:FK_ex} may not have finite first or second moments if $c({\bf x})$ is allowed to be positive somewhere in $\Omega$. Clearly, if $c({\bf x})$ is strictly positive---as in the Helmholtz equation---the distribution is unbounded and \eqref{eq:FK_ex} naturally breaks down. 

In addition to the Helmholtz equation having the ``wrong'' reaction coefficient sign, the Feynman-Kac formula does not apply with complex-valued BVPs such as \eqref{eq:H0}  (see, however, \cite{Funaki1979}\cite{Schlottmann1999}\cite{Budaev2001}.)

Seeking to overcome the unsuitability of \eqref{eq:H0} for the Feynman-Kac formula, we note that the Helmholtz equation can be written as
\begin{align}
\label{eq:naive_iter}		
	\Delta u({\bf x}) - \left(\alpha -k^2\right) u({\bf x})= f({\bf x})-\alpha u({\bf x}),
\end{align}
which would be \textit{formally} solvable by Feynman-Kac provided that $\alpha\geq k^2$. (The numerics are those of spatially bounded diffusions, see \cite{Milstein&Tretyakov04}\cite{eikonal}.) Of course, this is impossible since the right-hand side contains the solution itself, but it suggests a fixed-point approach to solving \eqref{eq:H0}. Therefore, we envision approximations of the form 
$u=\sum u_i$, where each iterate solves a real Poisson equation $\Delta u_i = F(u_{i-1})$ and is thus amenable to the Feynman-Kac formula---so that the iterate BVP for $u_i$ can be solved by PDD in a highly scalable way. To construct such a nested iteration, the solution to \eqref{eq:H0} must be separated into its real and imaginary parts, and the respective iterates interleaved. It turns out that---in order to enforce convergence---the iteration depends on whether both parts of the solution are connected by the PDE or by the BCs. The first situation occurs with a damping term in \eqref{eq:H0}, while the latter is brought about by the BCs on the scatterer's surface and on the computational boundary. The analysis of said iterations---by combining stochastic and algebraic arguments---is the main contribution of this paper. 

\paragraph{Overview of the paper.} 

In Section 2, we specify and analyse the iterations described above, proving their convergence (within a range of $k$). First, we briefly review the necessary Feynman-Kac formulae for reference. We consider two Helmholtz BVPs, posed inside bounded domains with and without a hole and with BCs of Dirichlet and Robin type. The underlying physical situations of the BVPs addressed in Section 2 are wave propagation inside a cavity and the scattering of a plane wave.  
In either case, we design and study the convergence of a different iteration of Poisson BVPs. 
We provide a rigorous sufficient condition for convergence and a semi-heuristic necessary condition. Our experiments are restricted to the annular domain with mixed BCs (scattering), which is more suitable for PDD as its discretisation takes many DoFs and subdomains.

Section 3 assesses the theoretical findings of Section 2, and is divided into two parts. First, we focus on the feasibility results proved in Section 2, and explain how a scattering problem would be tackled with PDD and its potential of overcoming some of the numerical difficulties encountered by state-of-the-art methods {\em at low frequency}. Then, we explain how the same theoretical results preclude PDD from solving wave scattering problems at high frequencies. Consequently, our approach is currently limited to tamer problems at low frequency, for which there are better alternatives (than PDD). For that reason, this paper does not contain PDD simulations.

\section{The iterations}\label{S:Iters}
\subsection{Preliminaries: Feynman-Kac formula}\label{SS:FKSS:Validation}

Following \cite{2015fkformula}, one is able to represent the solution of elliptic BVPs with mixed boundary conditions via probabilistic calculus. Consider the open domain $\Omega \subset \bR^D $ with boundary $\partial \Omega$ such that $\partial \Omega$ is given by two (closed) subsets $ \partial \Omega_1  $ and $ \partial \Omega_2$. Over the closed domain ${\overline\Omega}=\Omega \cup \partial \Omega_1 \cup \partial \Omega_2$ we have the following BVP:
\begin{align}
    \label{eq: general pde}
    \begin{cases}
        \Delta u({\bf x})+ c({\bf x}) u({\bf x}) =f({\bf x}),
        &  {\bf x}\in  \Omega,
        \\
        u({\bf x})=\phi({\bf x}), &  {\bf x}\in \partial \Omega_1,
        \\
         \frac{\partial u}{\partial \bn}- \varphi({\bf x}) u({\bf x})=g({\bf x}), &   {\bf x}\in \partial \Omega_2, 
    \end{cases}
\end{align}
where  $c,\varphi: \bR^D \rightarrow \bR$ are nonnegative functions; $f,g,\phi: \bR^D \rightarrow \bR $;  
$\bn( {\bf x})\in {\mathbb R}^D$ is the unit outward  normal vector on $ \partial \Omega_2$;
$ \frac{\partial u}{\partial \bn}=   \bn({\bf x})^T \nabla u({\bf x})$ stands for the directional derivative; 
$\Delta$ is the Laplacian in ${\mathbb R}^D$;
the stochastic process $({\bf Y^{\bf y}}_t)_{t\geq 0}$ is the solution to the stochastic differential equation (SDE) 
\begin{align}
\label{eq:def_Y_t}
    \dd  {\bf Y}^{\bf y}_t&=\sqrt{2} \dd {\bf W}_t-\bn({\bf Y}^{\bf y}_t) \dd \xi_t, \quad
    {\bf Y}^{\bf y}_0=  {\bf y}\in \Omega, 
\end{align}
where ${\bf W}_t$ is a $D$-dimensional Brownian motion; and the local time process $\xi_t$ measures the time that $({\bf Y}^{\bf y})_{t\geq 0}$ spends on the boundary $ \partial \Omega_2$ up to time $t$, and is defined in \cite{levylocaltime} as 
\begin{align}
\nonumber
    \xi_t = \lim_{\varepsilon \rightarrow 0^+} \frac{\int_0^t \1_{\{ {\bf Y}^{\bf y}_s \in \Omega_{\varepsilon}\}} ds }{\varepsilon}
    \qquad \textrm{with}\qquad 
    \Omega_{\varepsilon}=\{ {\bf x}\in \Omega: d( {\bf x},\partial \Omega_2 )< \varepsilon\},~\varepsilon>0,
\end{align}
 where $\1_{\{\cdot\}}$ denotes the indicator function and $d(\cdot,\cdot)$ denotes the Euclidean distance.

Then, the Feynman-Kac formula (under certain assumptions) allows to represent the pointwise solution $u({\bf y})$ of \eqref{eq: general pde} in terms of an expectation over the It\^o diffusion process $({\bf Y}^{\bf y}_t)_{t\geq 0}$, namely 
\begin{align}
\label{eq: feynmann-kac for complete pde}
 u({\bf y})
 =
 &
 \bE\left[-\int_0^{\tau^{\bf y}} f\left({\bf Y}^{\bf y}_t\right) \Pi_t ~\dd t \right]
 +\bE\Big[
  \phi\left(Y_{ \tau^{\bf y}}\right)  \Pi_{\tau^{\bf y}}  \Big]
  +
 \bE\left[\int_0^{\tau^{\bf y}} g\left({\bf Y}^{\bf y}_t\right)  \Pi_t \dd\xi_t\right], 
 \\
\nonumber 
 \quad 
 \textrm{where}\quad
 \Pi_t= &\exp \left(\int_0^{t} c\left({\bf Y}_s\right)\dd s+ \varphi({\bf Y}_s^{\bf y})   \dd \xi_s \right),
\end{align}
with $\tau^{\bf y}=\inf \left\{t\geq 0: {\bf Y}^{\bf y}_t \in \partial \Omega_1\right\}$ being the first hitting time of ${\bf Y}^{\bf y}_t$ to the boundary $\partial \Omega_1$. Intuitively, the process starts at ${\bf x}={\bf y} \in \Omega$, is normally reflected on $\partial\Omega_2$, and terminates upon hitting $\partial\Omega_1$.

\subsection{Cavity domain}
\label{SS:Cavity}
We consider the following Helmholtz BVP on a simply  connected domain,
\begin{align}
    \label{eq: origin Helmholtz interior}
        \begin{cases}
        \Delta u+(ip+k^2) u =f^{\real}({\bf x})+if^{\imag}({\bf x}),
        &{\bf x}\in  \Omega,
        \\
        u=g^{\real}({\bf x})+ig^{\imag}({\bf x}), &{\bf x}\in\partial \Omega_1,
        \\
        \frac{\partial u}{\partial \bn} =0 , &{\bf x}\in \partial \Omega_2,
    \end{cases}
\end{align}
where $k,p$ are real, positive constants and $f^{\real},f^{\imag},g^{\real},g^{\imag}: \bR^D \rightarrow \bR$ are real scalar functions. We assume that the cavity walls may have a reflecting portion $\partial\Omega_2$, while $\partial\Omega_1\neq\emptyset$ for well-posedness.

Let us decompose $u=u^{\real}+i u^{\textrm{\imag}}$ into the real and imaginary components. From \eqref{eq: origin Helmholtz interior}, one can write a system of coupled PDEs for them: 
\begin{align}
    \label{eq:origin helmholtz int verion with real and im part}
    \begin{cases}
        \Delta u^{\real}+k^2 u^{\real} = f^{\real}({\bf x})+pu^{\imag}
        &, {\bf x}\in  \Omega,
        \\
        u^{\real}= g^{\real}({\bf x}) &, {\bf x}\in\partial \Omega_1,
        \\
        \frac{\partial u^{\real}}{\partial \bn} =0 &, {\bf x}\in \partial  \Omega_2 ,
    \end{cases}
    \quad 
    \begin{cases}
        \Delta u^{\imag}+k^2 u^{\imag} =f^{\imag}({\bf x})-p u^{\real}
        &, {\bf x}\in  \Omega,
        \\
        u^{\imag}= g^{\imag}({\bf x}) &, {\bf x}\in\partial \Omega_1,
        \\
        \frac{\partial u^{\imag}}{\partial \bn} =0 &,{\bf x}\in \partial  \Omega_2 .
    \end{cases}
\end{align}
Now, let $\alpha\geq k^2>0$ and consider the system of coupled PDEs (stemming from iterating \eqref{eq:origin helmholtz int verion with real and im part} in the context of the argument used in \eqref{eq:naive_iter})
\begin{align}
\label{eq: def of pde v2 int}
    \begin{cases}
        \Delta v_0-(\alpha-k^2) v_0  =f^{\real}({\bf x}),
        \\
        \Delta v_1 -(\alpha-k^2) v_1  =-\alpha v_0 + pw_0,
        \\
        \quad \vdots
        \\
        \Delta v_n -(\alpha-k^2) v_n  =-\alpha v_{n-1} + pw_{n-1},
        \\
        \quad \vdots
        \\
    \end{cases}
    ~
    \begin{cases}
        \Delta w_0-(\alpha-k^2) w_0  =f^{\imag}({\bf x}),
        \\
        \Delta w_1 -(\alpha-k^2) w_1  =-\alpha w_0-p v_0 ,
        \\
        \quad \vdots
        \\
        \Delta w_n -(\alpha-k^2) w_n  =-\alpha w_{n-1} -p v_{n-1}\\
        \quad \vdots
        \\
    \end{cases}
\end{align}
with boundary conditions for the $\{v_n\}_n$ and the $\{w_n\}_n$ given by
\begin{align}
\label{eq: bc1 of pde v2 int}
  \textrm{if }{\bf x}&\in  \partial \Omega_1,\quad 
    \begin{cases}
          v_0  =g^{\real}({\bf x}),
        \\
         v_n =0, ~n\geq 1,
    \end{cases}
    \quad \textrm{and} \quad 
    \begin{cases}
          w_0  =g^{\imag}({\bf x}),
        \\
         w_n =0, ~n\geq 1,
    \end{cases}
\end{align}
and 
\begin{align}
\label{eq: bc2 of pde v2 int}
   \textrm{if } {\bf x}&\in  \partial \Omega_2,
    \quad 
     \frac{\partial v_n}{\partial \bn}=  \frac{\partial w_n}{\partial \bn} = 0, \quad n\geq 0.
\end{align} 
 
The next result establishes the link between the nested iteration system \eqref{eq: def of pde v2 int}-\eqref{eq: bc2 of pde v2 int} and the cavity problem \eqref{eq: origin Helmholtz interior}.

\begin{remark}
\label{rem:Same tau-y across n} 

Inspecting \eqref{eq: def of pde v2 int}, we see that across $n$, the left-hand side (LHS) of the Equation \eqref{eq: def of pde v2 int} for $v_n$ and $w_n$ do not change albeit its right-hand side (RHS) do.  This means that at the level of the stochastic representation in Equation \eqref{eq: general pde}-\eqref{eq: feynmann-kac for complete pde}, $\tau^{\bf y}$ is different (as it uses a different independent Brownian motion at each $n$) but identically distributed across $n$, and thus $\bE[\tau^{\bf y}]$ is the same for all $n$.    
\end{remark}

If $\alpha\geq k^2$ and 
$(\alpha+p)\bE[\tau^{\bf y}]<1$ for any ${\bf y}\in\Omega$ (which depends on the shape of $\Omega$, $\partial\Omega_1$, and $\partial\Omega_2$), the sum of iterates converges to the solution of the Helmholtz equation in the cavity.  

 \begin{theorem}
\label{lemma: second version conv int}
    Let $v_n,w_n$ be the solutions of the PDEs defined in \eqref{eq: def of pde v2 int} with boundary conditions \eqref{eq: bc1 of pde v2 int}-\eqref{eq: bc2 of pde v2 int} for $n\in \{0\}\cup{\mathbb N}$. 
Let $\|v_n\|_\infty:=\sup_{{\bf x}\in \Omega} |v_n({\bf x})|$ and $\|w_n\|_\infty:=\sup_{{\bf x}\in \Omega} |w_n({\bf x})|$.  

If $  (\alpha+p)\bE[\tau^{\bf y}]    <1$ for all ${\bf y}\in \Omega$, then  $\lim_{n\rightarrow\infty} \|v_n\|_\infty=0$ and $\lim_{n\rightarrow\infty} \|w_n\|_\infty=0$, with  $\sum_{n=0}^\infty \|v_n\|_\infty <\infty$ and $\sum_{n=0}^\infty \|w_n\|_\infty <\infty$. Moreover, $\sum_{n=0}^\infty v_n=u^{\real}$ and $\sum_{n=0}^\infty w_n=u^{\imag}$, where $u^{\real},~u^{\imag} $ are as in \eqref{eq:origin helmholtz int verion with real and im part}.
\end{theorem}

\begin{proof}
Under the assumptions, from the Feynman-Kac formula \eqref{eq: feynmann-kac for complete pde}, for all ${\bf y}\in\Omega$ and $n=0$, 
    \begin{align*}
        |v_0({\bf y})|&= \Big| 
        \bE \Big[  -\int_0^{\tau^{\bf y}}  f^{\real}({\bf Y}_s^{\bf y})  e^{-(\alpha-k^2)s} \dd s
        \Big]
        +
        \bE \big[  -e^{-(\alpha-k^2) \tau^{\bf y}}
        g^{\real}(Y_{\tau^{\bf y}}^{\bf y})   
        \big]  \Big| 
        \\
        &\leq \Big| 
        -\bE \big[\tau^{\bf y} \big]\inf_{y\in   \Omega } f^{\real}({\bf y}) +
        \sup_{y\in \partial \Omega_1} g^{\real}({\bf y}) \Big| <\infty,
        \\
        |w_0({\bf y})|&= \Big| 
        \bE \Big[  -\int_0^{\tau^{\bf y} }  f^{\imag}( {\bf Y }_s^{\bf y})  e^{-(\alpha-k^2)s} \dd s
        \Big]
        +
        \bE \big[  e^{-(\alpha-k^2) \tau^{\bf y}}
         g^{\imag}(Y_{\tau^{\bf y}}^{\bf y})
        \big]  \Big| 
        \\
        &\leq  \Big| 
        -\bE \big[\tau^{\bf y} \big]\inf_{y\in   \Omega } f^{\imag}({\bf y}) +
        \sup_{y\in \partial \Omega_1} g^{\imag}({\bf y}) \Big| 
        <\infty.
    \end{align*}
    For all $n\geq 1$, ${\bf y}\in \Omega$, we have 
    \begin{align*}
        v_n({\bf y})&=\bE \Big[  \int_0^{\tau^{\bf y} } \big(\alpha v_{n-1} ( {\bf Y }_s^{\bf y}) 
        - p w_{n-1} ( {\bf Y }_s^{\bf y})
        \big) e^{-(\alpha-k^2)s} \dd s
        \Big],
    \\
        w_n({\bf y})&=\bE \Big[  \int_0^{\tau^{\bf y}}  \big(\alpha w_{n-1} ( {\bf Y }_s^{\bf y}) 
        + p v_{n-1} ( {\bf Y }_s^{\bf y}) \big)
        e^{-(\alpha-k^2)s} \dd s
        \Big].
    \end{align*}
    Taking maximum norm on both sides,
     \begin{align*}
     \|v_n \|_\infty
        &\leq \|v_{n-1} \|_\infty\bE \Big[  \int_0^{\tau^{\bf y}} \alpha    e^{-(\alpha-k^2)s} \dd s
        \Big]
        +
        \|w_{n-1} \|_\infty\bE \Big[  \int_0^{\tau^{\bf y}} p    e^{-(\alpha-k^2)s} \dd s
        \Big]
        \\
         &\leq (\alpha\|v_{n-1} \|_\infty + p \|w_{n-1} \|_\infty )\bE \big[\tau^{\bf y}\big] 
        .
     \end{align*}
    Similarly, for $ \|w_n \|_\infty$, we have
     \begin{align*}
        \|w_n \|_\infty
         &\leq 
         (\alpha\|w_{n-1} \|_\infty + p \|v_{n-1} \|_\infty )\bE \big[\tau^{\bf y}\big]
         \end{align*}
      and therefore,   \begin{align*}
         \big( \|v_n \|_\infty+\|w_n \|_\infty\big)
         &\leq 
         \big( \|v_{n-1} \|_\infty+\|w_{n-1} \|_\infty\big) 
         ~( \alpha + p) \bE \big[\tau^{\bf y}\big]
        .
     \end{align*}
Since $( \alpha + p) \bE \big[\tau^{\bf y}\big] <1$ for all ${\bf y} \in \Omega$,   we deduce that  $\lim_{n\rightarrow\infty}\|v_n\|_\infty=0$ and  
$\lim_{n\rightarrow\infty}\|w_n\|_\infty=0$.
Next, we show that the series of $v_n,w_n$ converge to the exact solution in \eqref{eq:origin helmholtz int verion with real and im part}. 
We first show that the sum of the supremum norms is finite (employing the formula for the sum of a geometric series),
\begin{align*}
    \sum_{n=0}^{\infty}  \big( \|v_n \|_\infty+\|w_n \|_\infty\big)
    &
    \leq
    \big( \|v_0 \|_\infty+\|w_0 \|_\infty\big)  \sum_{n=0}^{\infty}
    \big(( \alpha + p)  \bE \big[\tau^{\bf y} \big]   \big)^n
    \\
    &
    \leq
   \big( \|v_0 \|_\infty+\|w_0 \|_\infty\big)  \frac{1}{1-  ( \alpha + p) \bE \big[\tau^{\bf y} \big]     }<\infty,
\end{align*}
Next, we consider $\sum_{n=0}^N v_n$ and $\sum_{n=0}^N w_n$ with $N\geq 1$. From \eqref{eq: def of pde v2 int}-\eqref{eq: bc2 of pde v2 int}, we have
\begin{align}
    \nonumber
    &\begin{cases}
        \Delta \sum_{n=0}^N v_n-(\alpha-k^2) \sum_{n=0}^N v_n =f^{\real}({\bf x})-\alpha \sum_{n=0}^{N-1} v_n+ p \sum_{n=0}^{N-1} w_n,
        &{\bf x}\in  \Omega,
        \\
        \sum_{n=0}^N v_n=g^{\real}({\bf x}), &{\bf x}\in\partial  \Omega_1,
        \\
        {\partial \sum_{n=0}^N v_n}/{\partial \bn} =0, &{\bf x}\in \partial \Omega_2.
    \end{cases}
    \\
    \nonumber
    &\begin{cases}
        \Delta \sum_{n=0}^N w_n-(\alpha-k^2) \sum_{n=0}^N w_n =f^{\imag}({\bf x})-\alpha \sum_{n=0}^{N-1} w_n
        - p \sum_{n=0}^{N-1} v_n
        ,
        &{\bf x}\in  \Omega,
        \\
        \sum_{n=0}^N w_n=g^{\imag}({\bf x}), &{\bf x}\in\partial  \Omega_1,
        \\
         {\partial \sum_{n=0}^N w_n}/{\partial \bn} =0, &{\bf x}\in \partial \Omega_2.
    \end{cases}
\end{align}
Letting $N\rightarrow \infty$ and rearranging terms, we have

\begin{align*}
    &\begin{cases}
        \Delta \sum_{n=0}^\infty v_n+k^2 \sum_{n=0}^\infty v_n
        - p \sum_{n=0}^\infty w_n
        = f^{\real}({\bf x})+\lim_{N\rightarrow\infty} (v_N-w_N)  
        =f^{\real}({\bf x}),
        &{\bf x}\in  \Omega,
        \\
        \sum_{n=0}^\infty v_n=g^{\real}({\bf x}), &{\bf x}\in\partial  \Omega_1,
        \\
        {\partial \sum_{n=0}^N v_n}/{\partial \bn} =0, &{\bf x}\in \partial \Omega_2,
    \end{cases}
    \\
    &\begin{cases}
        \Delta \sum_{n=0}^\infty w_n+k^2 \sum_{n=0}^\infty w_n 
        + p \sum_{n=0}^\infty v_n
        = f^{\imag}({\bf x})+\lim_{N\rightarrow\infty} (v_N+w_N)  
        =f^{\imag}({\bf x}),
        &{\bf x}\in  \Omega,
        \\
        \sum_{n=0}^\infty w_n=g^{\imag}({\bf x}), &{\bf x}\in\partial  \Omega_1,
        \\
         {\partial \sum_{n=0}^N w_n}/{\partial \bn} =0, &{\bf x}\in \partial \Omega_2,
    \end{cases}
\end{align*}
where we used the fact that $\lim_{N\rightarrow\infty}\|v_N\|_\infty=0$ and  
$\lim_{N\rightarrow\infty}\|w_N\|_\infty=0$ and the series convergence of $v_n$ and $w_n$.  
Thus, we conclude that $\sum_{n=0}^\infty v_n $ and $\sum_{n=0}^\infty w_n  $ recover the PDE system for $u^{\real}$ and $u^{\imag}$ in \eqref{eq:origin helmholtz int verion with real and im part}. 

\end{proof}

\begin{remark}
    In general, the problem described in  \eqref{eq:origin helmholtz int verion with real and im part} is ill-posed for some choices of $f^{\real}({\bf x}),f^{\imag}({\bf x})$ and $k$---see \cite{illposedproblembook}.  The same situation carries over to the sum of iterations \eqref{eq: def of pde v2 int}.  
\end{remark}

\begin{remark}
The sum of iterations $\sum_{n=0}^{\infty}v_n+i\sum_{n=0}^{\infty}w_n$ may fail to converge to the solution of \eqref{eq:origin helmholtz int verion with real and im part} (assuming that the latter exists) if $k> k_*^C$, given by  
\begin{align}
    \label{eq:k_*^C}
k_*^C=  \sqrt{\frac{1}{\max_{{\bf y}\in\Omega}\bE[\tau^{\bf y}]}-p}.
\end{align}
This is because there is no $\alpha>(k_*^C)^2$ such that the assumptions of Theorem \ref{lemma: second version conv int} can hold.
\end{remark}

\begin{remark}
It can be seen by direct application of the Feynman-Kac formula that $\bE[\tau^{\bf x}]=E({\bf x})$, where $E({\bf x})$ solves the BVP
\begin{align}
    \label{eq: E[tau]}
    \begin{cases}
        \Delta E({\bf x}) =-1,
        &  {\bf x}\in  \Omega,
        \\
        E({\bf x})=0, &  {\bf x}\in \partial \Omega_1,
        \\
         \frac{\partial E}{\partial \bn}=0, &   {\bf x}\in \partial \Omega_2.
    \end{cases}
\end{align}
This is the so-called (and well-known) mean-hitting time problem \cite{Milstein&Tretyakov04}.
\end{remark}

\subsection{Annular domain}\label{SS:Annular}
Here, we consider the Helmholtz equation on a disk-shaped domain with a hole inside. 
The boundary condition is of  Dirichlet type (i.e. absorbing impinging Feynman-Kac diffusions) on the inner boundary, $\partial\Omega_1$, and of Robin type (i.e. reflecting diffusions) on the outer edge of the disk, $\partial\Omega_2$. 
The BVP is given by 
\begin{align}
    \label{eq: origin helmholtz}
        \begin{cases}
        \Delta u+k^2 u =0,
        &{\bf x}\in  \Omega,
        \\
        u=-e^{-i k{ x_1}}, &{\bf x}\in\partial \Omega_1,\quad {\bf x} = (x_1,x_2),\\
        \frac{\partial u}{\partial \bn}-ik u=0 , &{\bf x}\in \partial \Omega_2,
    \end{cases}
\end{align}

As before, we choose to rewrite \eqref{eq: origin helmholtz} as two coupled BVPs for the real and imaginary parts:
\begin{align}
    \label{eq:origin helmholtz with real and im part}
    \begin{cases}
        \Delta u^{\real}+k^2 u^{\real} =0
        &, {\bf x}\in  \Omega,
        \\
        u^{\real}=-\cos(k{ x_1}) &, {\bf x}\in\partial \Omega_1,\\
        \frac{\partial u^{\real}}{\partial \bn}+k u^{\imag}=0  &, {\bf x}\in \partial \Omega_2,
    \end{cases}
    \qquad\textrm{and}\qquad 
    \begin{cases}
        \Delta u^{\imag}+k^2 u^{\imag} =0
        &, {\bf x}\in  \Omega,
        \\
        u^{\imag}=\sin(k{ x_1}) &, {\bf x}\in\partial \Omega_1,\\
        \frac{\partial u^{\imag} }{\partial \bn}-k u^{\real}=0 &, {\bf x}\in \partial \Omega_2.
    \end{cases}
\end{align}
 
Setting again $\alpha\geq k^2>0$, we propose (as in Section \ref{SS:Cavity}) the following iteration of BVPs:
\begin{align}
\label{eq: def of pde v2}
    {\bf x}\in  \Omega,
    \begin{cases}
        \Delta v_0-(\alpha-k^2) v_0  =0,
        \\
        \Delta v_1 -(\alpha-k^2) v_1  =-\alpha v_0 ,
        \\
        \cdots
        \\
        \Delta v_n -(\alpha-k^2) v_n  =-\alpha v_{n-1} ,
    \end{cases}
    \quad
    \begin{cases}
        \Delta w_0-(\alpha-k^2) w_0  =0,
        \\
        \Delta w_1 -(\alpha-k^2) w_1  =-\alpha w_0 ,
        \\
        \cdots
        \\
        \Delta w_n -(\alpha-k^2) w_n  =-\alpha w_{n-1},
    \end{cases}
\end{align}
with boundary conditions
\begin{align}
\label{eq: bc1 of pde v2}
  {\bf x}\in  \partial \Omega_1,
    \begin{cases}
          v_0  =-\cos (k{ x_1}),
        \\
         v_n =0, ~n\geq 1,
    \end{cases}
    \qquad \quad 
    \begin{cases}
          w_0  =\sin (k{ x_1}),
        \\
         w_n =0, ~n\geq 1,
    \end{cases}
\end{align}
and 
\begin{align}
\label{eq: bc2 of pde v2}
  {\bf x}\in   \partial \Omega_2,
    \begin{cases}
         \frac{\partial v_0}{\partial \bn}=0,
        \\
        \frac{\partial v_n}{\partial \bn}=-k w_{n-1},  ~n\geq 1,
    \end{cases}
    \qquad \quad 
    \begin{cases}
         \frac{\partial w_0}{\partial \bn}=0,
        \\
        \frac{\partial w_n}{\partial \bn}=k v_{n-1},  ~n\geq 1.
    \end{cases}
\end{align}

The next result establishes the convergence of the sum of iterates to the solution of the original Helmholtz BVP. 
\begin{remark} 
Inspecting \eqref{eq: def of pde v2}, we see that across $n$, the  LHS  of the Equation \eqref{eq: def of pde v2} for $v_n$ and $w_n$ does not change albeit its  RHS  is. We are interested only in $\bE[\tau^{\bf y} ]$ and $ \bE[\xi^{\bf y} ]$, thus the quantities of interest is the same across $n$. This is the same argument as in Remark \ref{rem:Same tau-y across n} above.    
\end{remark}

\begin{theorem}
\label{lemma: second version conv}
    Let $\Omega,\partial\Omega_1,\partial\Omega_2$ and $v_n,w_n$ be defined in \eqref{eq: def of pde v2} with boundary conditions \eqref{eq: bc1 of pde v2} \eqref{eq: bc2 of pde v2} for $n\in \{0\}\cup{\mathbb N}$. 
    
    If $(\alpha\bE[\tau^{\bf y} ]+ k \bE[\xi^{\bf y} ])   <1$ for all ${\bf y}\in \Omega$, then  $\lim_{n\rightarrow\infty} \|v_n\|_\infty=0$ and $\lim_{n\rightarrow\infty} \|w_n\|_\infty=0$, with  $\sum_{n=0}^\infty \|v_n\|_\infty <\infty$ and $\sum_{n=0}^\infty \|w_n\|_\infty <\infty$. Moreover, $\sum_{n=0}^\infty v_n=u^{\real}$ and $\sum_{n=0}^\infty w_n=u^{\imag} $, where $u^{\real},~u^{\imag} $ are defined in \eqref{eq:origin helmholtz with real and im part}.

\end{theorem}

\begin{proof}
    
    
    Under the assumption, from the Feynman-Kac formula \eqref{eq: feynmann-kac for complete pde},   for all ${\bf y}\in\Omega$ and $n=0$, we have 
    \begin{align*}
        |v_0({\bf y})|&=\bE \big[  |-e^{-(\alpha-k^2) \tau^{\bf y}}
        \cos (k  {Y}_{\tau^{\bf y},1}^{\bf y} )|   
        \big] \leq 1,
        \quad 
        {\bf Y}_{\tau^{\bf y}}^{\bf y}
        = ( {Y}_{\tau^{\bf y},1}^{\bf y}, {Y}_{\tau^{\bf y},2}^{\bf y})
        \\
        |w_0({\bf y})|&=\bE \big[  |e^{-(\alpha-k^2) \tau^{\bf y}}
        \sin (k {Y}_{\tau^{\bf y},1}^{\bf y} )|
        \big] \leq 1.
    \end{align*}
    For all $n\geq 1$, ${\bf y}\in \Omega$, we have 
    \begin{align*}
        v_n({\bf y})&=\bE \Big[  \int_0^{\tau^{\bf y} } \alpha v_{n-1} ({\bf Y}_s^{\bf y})  e^{-(\alpha-k^2)s} \dd s
        \Big]
        +
        \bE \Big[  
        \int_0^{\tau^{\bf y} }-k w_{n-1} ({\bf Y}_s^{\bf y})  e^{-(\alpha-k^2)s} d \xi_s
        \Big],
    \\
        w_n({\bf y})&=\bE \Big[  \int_0^{\tau^{\bf y} } \alpha w_{n-1} ({\bf Y}_s^{\bf y})  e^{-(\alpha-k^2)s} \dd s
        \Big]
        +
        \bE \Big[  
        \int_0^{\tau^{\bf y} }k v_{n-1} ({\bf Y}_s^{\bf y})  e^{-(\alpha-k^2)s} d \xi_s
        \Big].
    \end{align*}
    Taking maximum norm $\|v_n\|_\infty:=\sup_{{\bf x}\in \Omega} |v_n({\bf x})|$ on both sides yields
     \begin{align*}
     \|v_n \|_\infty
        &\leq \|v_{n-1} \|_\infty\bE \Big[  \int_0^{\tau^{\bf y} } \alpha    e^{-(\alpha-k^2)s} \dd s
        \Big]
        + \| w_{n-1}\|_\infty 
        \bE \Big[  \int_0^{\tau^{\bf y} } k   e^{-(\alpha-k^2)s} d \xi_s
        \Big]
        \\
         &\leq \alpha\|v_{n-1} \|_\infty 
        \bE \big[\tau^{\bf y} \big]
        + k\| w_{n-1}\|_\infty 
        \bE[\xi^{\bf y} ].
     \end{align*}
    Similarly, for $ \|w_n \|_\infty$, we have
     \begin{align*}
        \|w_n \|_\infty
         &\leq \alpha\|w_{n-1} \|_\infty 
        \bE \big[\tau^{\bf y} \big]
        + k\| v_{n-1}\|_\infty 
        \bE[\xi^{\bf y} ].
     \end{align*}
    To summarize, upon iteration we have 
     \begin{align}
     \nonumber
     (\|v_n \|_\infty+     \|w_n \|_\infty)
         &\leq 
         (\|v_{n-1} \|_\infty+     \|w_{n-1} \|_\infty)
         (\alpha  \bE \big[\tau^{\bf y} \big]    + k\bE[\xi^{\bf y} ]) 
        \\
        \label{eq: proof supremum norm decay}
         &\leq 
         (\|v_{0} \|_\infty+     \|w_{0} \|_\infty)
         \big(\alpha  \bE \big[\tau^{\bf y} \big]    + k\bE[\xi^{\bf y} ]\big)^{n}.
    \end{align}
Since $(\alpha  \bE \big[\tau^{\bf y} \big]    + k\bE[\xi^{\bf y} ])<1$ for all ${\bf y} \in \Omega$,   we have that $\lim_{n\rightarrow\infty}\|v_n\|_\infty=0$ and  
$\lim_{n\rightarrow\infty}\|w_n\|_\infty=0$.
We shall show that the summations of $v_n,w_n$ converge to the exact solution components in \eqref{eq:origin helmholtz with real and im part}.
We first show the summation of the supremum norm is finite:
\begin{align*}
    \sum_{n=0}^{\infty} (\|v_n \|_\infty+     \|w_n \|_\infty)
    &
    \leq
    (\|v_{0} \|_\infty+     \|w_{0} \|_\infty) \sum_{n=0}^{\infty}
    \big(\alpha  \bE \big[\tau^{\bf y} \big]    + k\bE[\xi^{\bf y} ]\big)^n
    \\
    &
    \leq
    (\|v_{0} \|_\infty+     \|w_{0} \|_\infty) \frac{1}{1- (\alpha  \bE \big[\tau^{\bf y} \big]    + k\bE[\xi^{\bf y} ])}<\infty.
\end{align*}
Next, we consider $\sum_{n=0}^N v_n$ and $\sum_{n=0}^N w_n$ with $N\geq 1$. From \eqref{eq: def of pde v2},\eqref{eq: bc1 of pde v2} and \eqref{eq: bc2 of pde v2}, we have

\begin{align}
    \nonumber
    &\begin{cases}
        \Delta \sum_{n=0}^N v_n-(\alpha-k^2) \sum_{n=0}^N v_n =-\alpha \sum_{n=0}^{N-1} v_n,
        &{\bf x}\in  \Omega,
        \\
        \sum_{n=0}^N v_n=-\cos(k{ x_1}), &{\bf x}\in\partial \Omega_1,\\
         {\partial \sum_{n=0}^N v_n}/{\partial \bn} =-k  \sum_{n=0}^{N-1} w_n, &{\bf x}\in \partial \Omega_2,
    \end{cases}
    \\
    \nonumber
    &\begin{cases}
        \Delta \sum_{n=0}^N w_n-(\alpha-k^2) \sum_{n=0}^N w_n =-\alpha \sum_{n=0}^{N-1} w_n,
        &{\bf x}\in  \Omega,
        \\
        \sum_{n=0}^N w_n=\sin(k{ x_1}), &{\bf x}\in\partial \Omega_1,\\
         {\partial \sum_{n=0}^N w_n}/{\partial \bn} =-k  \sum_{n=0}^{N-1} v_n, &{\bf x}\in \partial \Omega_2.
    \end{cases}
\end{align}
Letting $N\rightarrow \infty$ and rearranging terms, we have

\begin{align}
    \nonumber
    &\begin{cases}
        \Delta \sum_{n=0}^\infty v_n+k^2 \sum_{n=0}^\infty v_n = \lim_{n\rightarrow\infty} v_n
        =0,
        &{\bf x}\in  \Omega,
        \\
        \sum_{n=0}^\infty v_n=-\cos(k{ x_1}), &{\bf x}\in\partial \Omega_1,\\
         {\partial \sum_{n=0}^\infty v_n}/{\partial \bn} +k  \sum_{n=0}^{\infty} w_n
        = \lim_{n\rightarrow\infty} w_n
        =0, &{\bf x}\in \partial \Omega_2,
    \end{cases}
    \\
    \nonumber
    &\begin{cases}
        \Delta \sum_{n=0}^\infty w_n+k^2 \sum_{n=0}^\infty w_n = \lim_{n\rightarrow\infty} w_n
        =0,
        &{\bf x}\in  \Omega,
        \\
        \sum_{n=0}^\infty w_n=\sin(kx_1), &{\bf x}\in\partial \Omega_1,\\
         {\partial \sum_{n=0}^\infty w_n}/{\partial \bn} +k  \sum_{n=0}^{\infty} v_n= \lim_{n\rightarrow\infty} v_n
        =0, &{\bf x}\in \partial \Omega_2,
    \end{cases}
\end{align}
where we used the fact that $\lim_{n\rightarrow\infty}\|v_n\|_\infty=0$ and  
$\lim_{n\rightarrow\infty}\|w_n\|_\infty=0$.
We conclude that $\sum_{n=0}^\infty v_n $ and $\sum_{n=0}^\infty w_n  $ recover the PDE system for $u^{\real}$ and $u^{\imag} $ in \eqref{eq:origin helmholtz with real and im part}.

\end{proof} 

\begin{remark}
Again, Theorem \ref{lemma: second version conv} provides a sufficiency condition for convergence that is no longer fulfilled if $k$ exceeds a geometry-dependent threshold $k_*^A$. For a point ${\bf x}\in\Omega$, the largest possible number $k={\hat k}({\bf x})$ verifying both $\alpha\geq {\hat k}^2({\bf x})$ and
$(\alpha\bE[\tau^{\bf x} ]+ k \bE[\xi^{\bf x} ])<1$ is given by $\alpha={\hat k}^2({\bf x})$ and
\begin{align*}
{\hat k}({\bf x})= \frac{-\bE[\xi^{\bf x} ]+\sqrt{\bE^2[\xi^{\bf x}] +4\bE[\tau^{\bf x} ]}}{2\bE[\tau^{\bf x} ]}.   
\end{align*}
Note that ${\hat k}({\bf x})>0$ since $\bE[\tau^{{\bf x}}]>0$ for ${\bf x}\in\Omega$. The threshold 
$k_*^A$ is therefore given by
\begin{align}
\label{eq:k_*^A} 
k_*^A=  \min\limits_{{\bf x}\in\Omega}
 \frac{-\bE[\xi^{\bf x} ]+\sqrt{\bE^2[\xi^{\bf x}] +4\bE[\tau^{\bf x} ]}}{2\bE[\tau^{\bf x} ]} 
.   
\end{align}
\end{remark}

\begin{remark}
By the Feynman-Kac formula, it can be seen that the expected local time on the reflecting boundary $\partial\Omega_2$ is $\bE[\xi^{\bf x}]=L({\bf x})$, where $L({\bf x})$ solves the BVP 
\begin{align}
    \label{eq: E[xi]}
    \begin{cases}
        \Delta L({\bf x}) =0,
        &  {\bf x}\in  \Omega,
        \\
        L({\bf x})=0, &  {\bf x}\in \partial \Omega_1,
        \\
         \frac{\partial L}{\partial \bn}=1, &   {\bf x}\in \partial \Omega_2.
    \end{cases}
\end{align}
\end{remark}

\subsection{Connection with linear algebra}\label{SS:Algebra}
\label{section: Connection to linear algebra}
We show next that the numerical discretisation of the previous iterations entails a necessary condition for their convergence (in addition to the derived sufficiency ones), whence an upper bound for $k$ arises. For the purpose of illustration, we consider the matrix generated by the finite difference (FD) discretisation of the annular problem in \eqref{eq: origin helmholtz} as the argument is more transparent with a strong (collocation) scheme such as FD than it is with a weak scheme such as finite elements (FE). 

The FD discretisation of the original BVP \eqref{eq: def of pde v2} leads to an approximate solution on the FD grid nodes ${\bf u}$, which solves the linear system 
\begin{align*}
	B(k) {\bf u}  = {\bf b}(k),
\end{align*}
where the matrix $B$ contains: the FD discretisations, the left-hand side of the Helmholtz PDE on the rows corresponding to the PDE nodes, as well as an FD discretisation of the Robin BC on the rows corresponding to the FD nodes on $\partial\Omega_2$. On the other hand, the FD nodes with Dirichlet BCs can be eliminated right away and transferred to the right-hand side vector ${\bf b}$. Additionally, ${\bf b}$ contains the values of the right-hand side functions of the PDE and the Robin BCs evaluated on the grid nodes.

We now tackle the iteration of \eqref{eq: def of pde v2}. Setting $\alpha= k^2$ for simplicity, the iteration takes the form
\begin{align}
&\Delta v_0= 0,\,\Delta v_1=-k^2 v_0,\ldots,\Delta v_n=-k^2 v_{n-1},&\textrm{ if }{\bf x}\in\Omega,\nonumber\\
&\Delta w_0= 0,\,\Delta w_1=-k^2 w_0,\ldots,\Delta w_n=-k^2 w_{n-1},&\textrm{ if }{\bf x}\in\Omega,\nonumber\\
&v_0=-\cos{kx_1},\,w_0=+\sin{kx_1},\,v_n=0=w_n \,(n\geq 1)&\textrm{ if }{\bf x}\in\partial\Omega_1,\nonumber\\
&\frac{\partial v_0}{\partial \bn}  = \frac{\partial w_0}{\partial \bn}=0,\,  \frac{\partial v_n}{\partial \bn} =-k w_{n-1},\, \frac{\partial w_n}{\partial \bn}=+k v_{n-1} \,(n\geq 1)&\textrm{ if }{\bf x}\in\partial\Omega_2.	\nonumber
\end{align}	
	
Arranging the FD grid nodes as ${\bf x}=({\bf x}',{\bf x}'')$, such that ${\bf x}''\in\partial\Omega_2$  are the nodes where the now Neumann BC is enforced and ${\bf x}'$ are the PDE nodes (recall that the known Dirichlet nodal values are eliminated from the system), the corresponding iterative nodal solutions for $n\geq 1$ are
\begin{align*}
	{ {\bf v}_n}= ({ {\bf v}_n'},\,{ {\bf v}_n''}) ,\qquad { {\bf w}_n}= ({ {\bf w}_n'},\,{ {\bf w}_n''}), 
\end{align*} 
which respectively solve the 
 linear systems	
\begin{align*}
	A 
\left(\begin{array}{c}
{ {\bf v}_n'} \\{}\\ { {\bf v}_n''}
\end{array}\right)
=
\left(\begin{array}{c}
	-k^2 { {\bf v}_{n-1}'} \\{}\\-k{ {\bf w}_{n-1}''} 
\end{array}\right)	
\end{align*}
and
\begin{align*}
	A
	\left(\begin{array}{c}
		{ {\bf w}_n'} \\ {} \\ { {\bf w}_n''}
	\end{array}\right)
	=
	\left(\begin{array}{c}
		-k^2 { {\bf w}_{n-1}'} \\{} \\+k { {\bf v}_{n-1}''}
	\end{array}\right)	
\end{align*}
for a real FD discretisation matrix $A$ discretising the Laplacian and the Neumann BC operators. Note that $A$ does not depend on $k$ or $n$; that the systems are coupled on the ${\bf x}''$ nodes only; and the sign change. We can rewrite the linear systems as 
\begin{align*}
	A
	\left(\begin{array}{c}
		{ {\bf v}_n'} \\ {} \\ { {\bf v}_n''}
	\end{array}\right)
	=
	\left(\begin{array}{ccc}
		-k^2 & & 0 \\ & \ddots & \\ 0 & & -k
	\end{array}\right)
	\left(\begin{array}{c}
		{ {\bf v}_{n-1}'} \\{} \\ { {\bf w}_{n-1}''}
	\end{array}\right)	
\end{align*}
and
\begin{align*}
	A 
	\left(\begin{array}{c}
		{ {\bf w}_n' }\\ {} \\{ {\bf w}_n''} 
	\end{array}\right)
	=
	\left(\begin{array}{ccc}
-k^2 & & 0\\ & \ddots & \\0 & & +k
	\end{array}\right)
	\left(\begin{array}{c}
		{ {\bf w}_{n-1}'} \\{} \\ { {\bf v}_{n-1}''}
	\end{array}\right).	
\end{align*}
The matrices on the right-hand side are diagonal. Let ${ {\bf u}_{n}}={ {\bf v}_n} + i{ {\bf w}_n}$. It holds: 
\begin{align*}
	A 
	\left(\begin{array}{c}
		{ {\bf u}_{n}'} \\ {} \\ { {\bf u}_{n}''}
	\end{array}\right)
	&=
	\left(\begin{array}{ccc}
		-k^2 & & 0 \\ & \ddots & \\ 0 & & +ik
	\end{array}\right)
	\left(\begin{array}{c}
		{ {\bf u}_{n-1}'} \\{} \\ { {\bf u}_{n-1}''} 
	\end{array}\right),	
\end{align*}
and therefore
\begin{align*}
     { {\bf u}_{n}} &= A^{-1} ~\left(\begin{array}{ccc}
		-k^2 & & 0 \\ & \ddots &  \\ 0 & & +ik
	\end{array}\right) { {\bf u}_{n-1}}
    =: G  { {\bf u}_{n-1}}.
\end{align*}
Let ${\bf u}_0$ be the nodal vector ${\bf v}_0+i{\bf w}_0$. By virtue of the matrix geometric series, 
\begin{align*}
	{\bf u }= \sum\limits_{n=0}^{\infty} G^n { {\bf u}_{0}} =  (I-G)^{-1} { {\bf u}_{0}} 
\end{align*}
which converges if and only if the spectral radius of $G$ is less than one. 

Figure \ref{fig: lines4} shows the spectral radius (left) and condition number (right) of the matrices $A^{-1}$ and $G$ as a function of $k$, in a numerical example solved with two FD grids of internodal separation $h=0.1$ and $h=0.05$. It can be seen that the spectral radius becomes larger than one for $k\gtrsim 1.2$ (with $h=0.1$) or $k\gtrsim 1$ (with $h=0.05$). (In both cases the maximum $k$ is over the theoretical value of $k_*^A$ for this problem, namely $k_*^A\approx 0.81$.) 

We point out in closing that $A$ is not positive definite owing to the reflecting BCs---but it would be with Dirichlet-only BCs.

\begin{figure}[h!bt]
    \centering
    \begin{subfigure}{.48\textwidth}
			\centering
 			\includegraphics[scale=0.48]{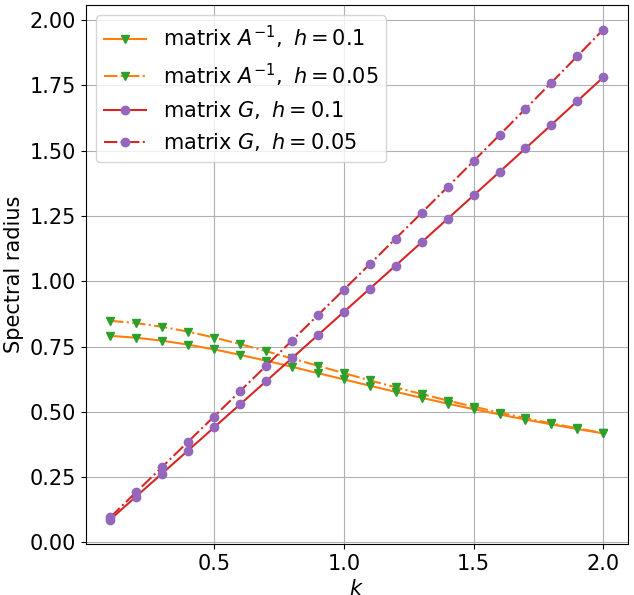}
		\end{subfigure}%
    \begin{subfigure}{.48\textwidth}
			\centering
 			\includegraphics[scale=0.48]{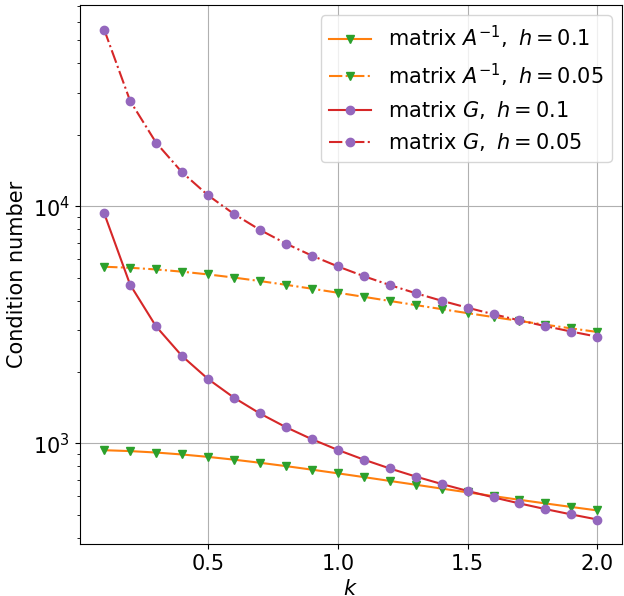}
		\end{subfigure}%
    \caption{ The spectral radius and condition number of the matrices generated by the finite difference method with $\alpha=k^2$ with gridsize $h\in\{0.1,~0.05\}$ applied on the problem \eqref{eq: origin helmholtz} on a square $[-0.5,0.5]\times[-0.5,0.5]$ with a square hole $[-0.15,0.15]\times[-0.15,0.15]$. 
     }
    \label{fig: lines4}
\end{figure}

\subsection{Numerical validation}\label{SS:Validation}

In this section, we consider three different annular shapes as examples of the problem in Section \ref{SS:Annular}. The shapes are graphically displayed in Figure \ref{fig:3-1} and described as follows
\begin{itemize}
    \item \textit{Shape 1} is the square $[-0.3,0.3]\times [-0.3,0.3]$ with a circular hole centred at $(0,0)$ and radius  $r=0.15$. 
    \item \textit{Shape 2} is a disk (centre at $(0.8,0.5)$ with radius  $r=0.45$) with a rhomboidal hole (centre at $(0,0)$ with length  $l=0.15$).
    \item \textit{Shape 3} is the rectangle $[-0.5,0.5]\times [-0.8,0.8]$ with a rectangular hole $[-0.05,0.05]\times [-0.4,0.4]$.
\end{itemize}
For all shapes, we computed an accurate numerical approximation to \eqref{eq: origin helmholtz} with FE (implemented by means of Matlab's {\tt PDETool}). It was used as a benchmark to investigate the new iterative scheme from \eqref{eq: def of pde v2},\eqref{eq: bc1 of pde v2},\eqref{eq: bc2 of pde v2}. The BVPs for those iterations as well as for $\bE[\tau^{\bf x}]$ \eqref{eq: E[tau]} and for $\bE[\xi^{\bf x}]$ \eqref{eq: E[xi]} were also calculated with FE on the three geometries. Convergence took place in all cases as predicted by Theorem \ref{lemma: second version conv}. 

\begin{figure}[h!bt]
 \centering
    \begin{subfigure}{.31\textwidth}
			\centering
 			\includegraphics[scale=0.42]{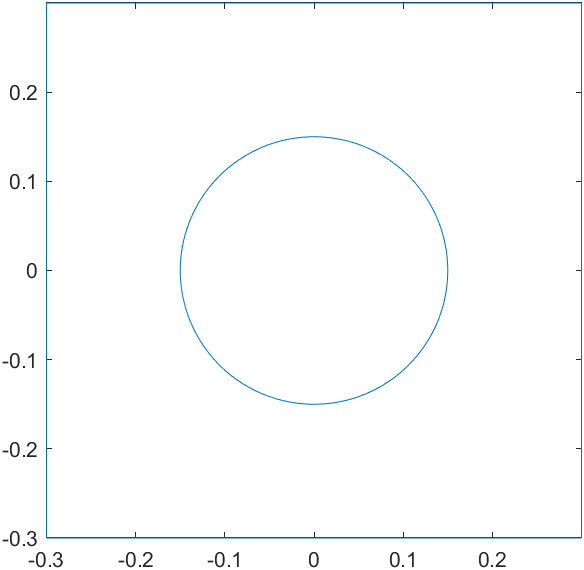}
			\caption{Shape 1: $k_*^A=1.93$}
		\end{subfigure}%
		\begin{subfigure}{.31\textwidth}
			\centering
 			\includegraphics[scale=0.42]{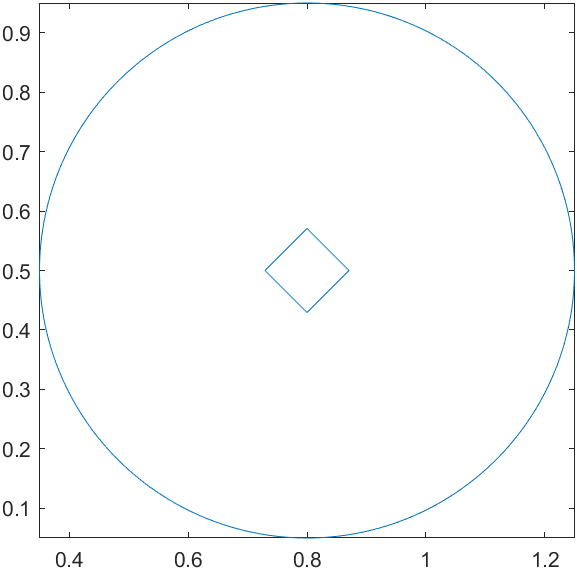}
			\caption{Shape 2: $k_*^A=0.95$}
		\end{subfigure}
            \begin{subfigure}{.31\textwidth}
			\centering
 			\includegraphics[scale=0.42]{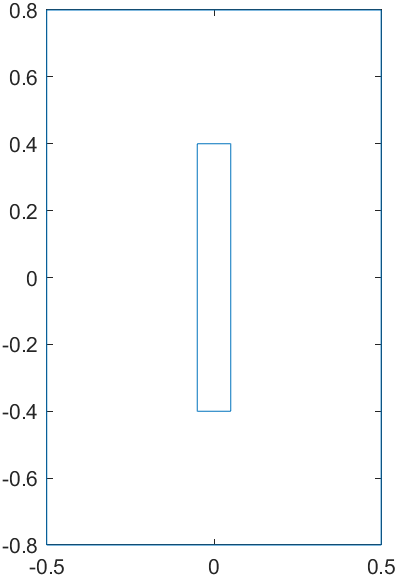}
			\caption{Shape 3: $k_*^A=1.24$}
		\end{subfigure}
    \caption{Geometry of three annular domains for the numerical experiments.}
    \label{fig:3shapes}
\end{figure}

An illustration is given in Figure \ref{fig:3-1}, where we show the real part of the iterates $v_n$ for $n\in\{0,5,10\}$ and the summation of the iterations $\sum_{n=0}^N v_n$ with $N=30$, as well as the difference (in absolute value) with respect to the true (numerical) solution.

\begin{figure}[h!bt]
 \centering
    \begin{subfigure}{.31\textwidth}
			\centering
 			\includegraphics[scale=0.35]{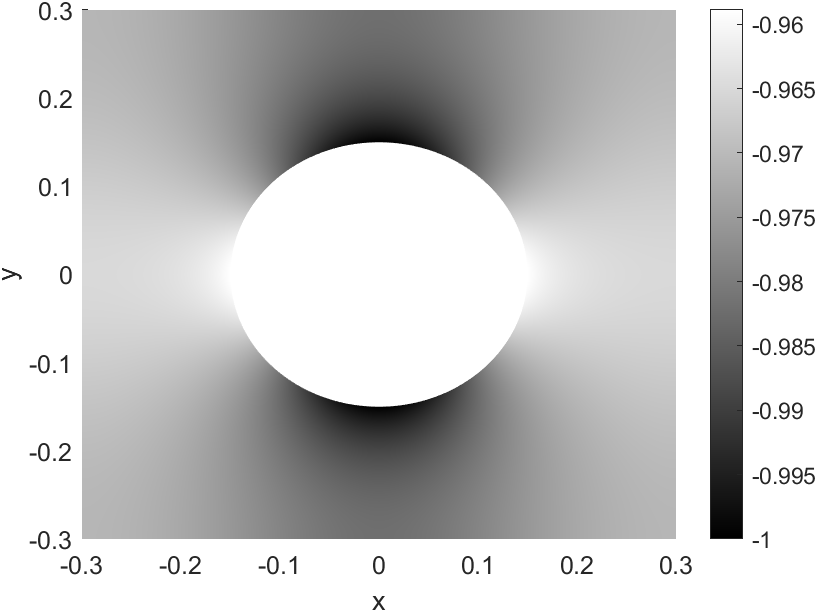}
			\caption{Iterate $v_0$}
		\end{subfigure}%
		\begin{subfigure}{.31\textwidth}
			\centering
 			\includegraphics[scale=0.35]{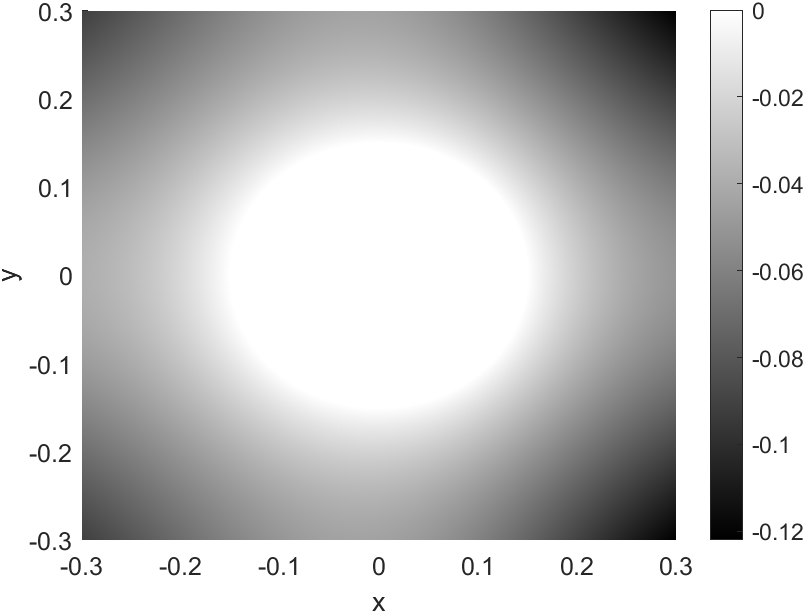}
			\caption{Iterate $v_5$}
		\end{subfigure}
            \begin{subfigure}{.31\textwidth}
			\centering
 			\includegraphics[scale=0.35]{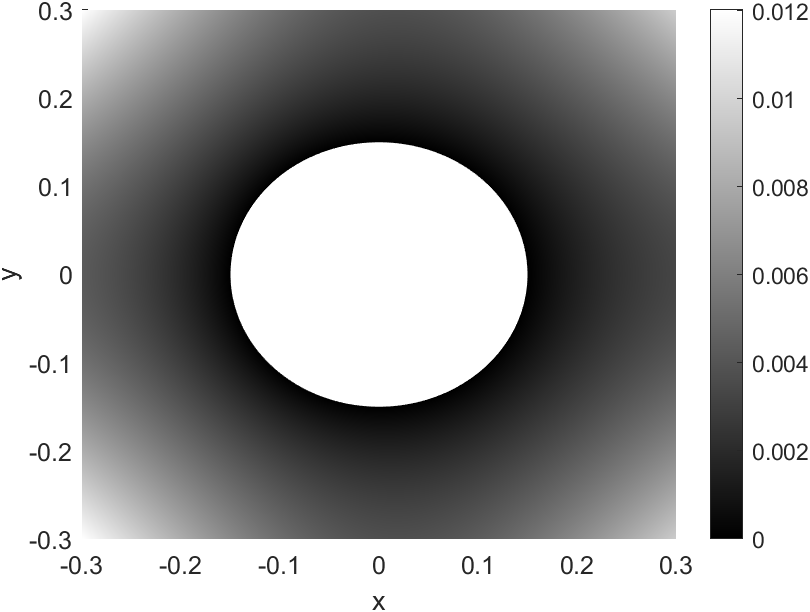}
			\caption{Iterate $v_{10}$}
		\end{subfigure}
  \\
            \begin{subfigure}{.31\textwidth}
			\centering
 			\includegraphics[scale=0.35]{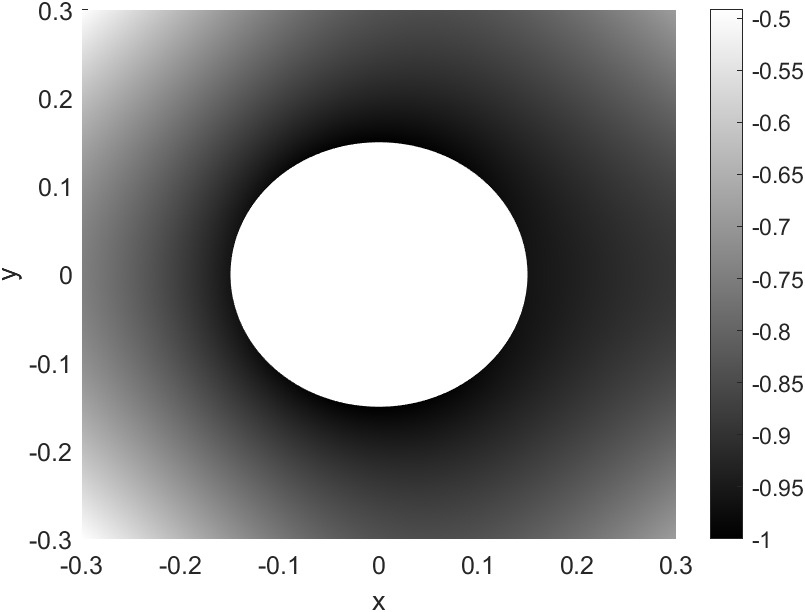}
			\caption{Sum of first $N=30$ iterates}
		\end{subfigure}
            \begin{subfigure}{.31\textwidth}
			\centering
 			\includegraphics[scale=0.35]{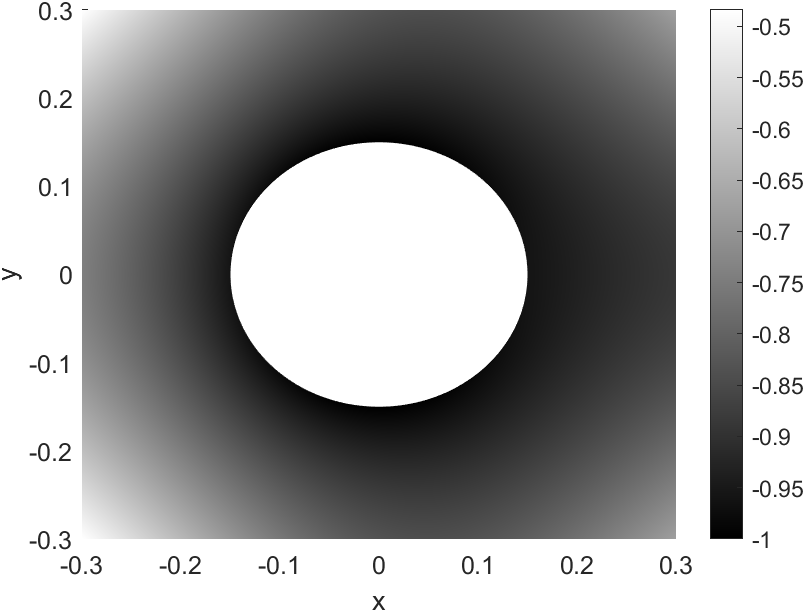}
			\caption{Solution for comparison}
		\end{subfigure}
            \begin{subfigure}{.31\textwidth}
			\centering
 			\includegraphics[scale=0.35]{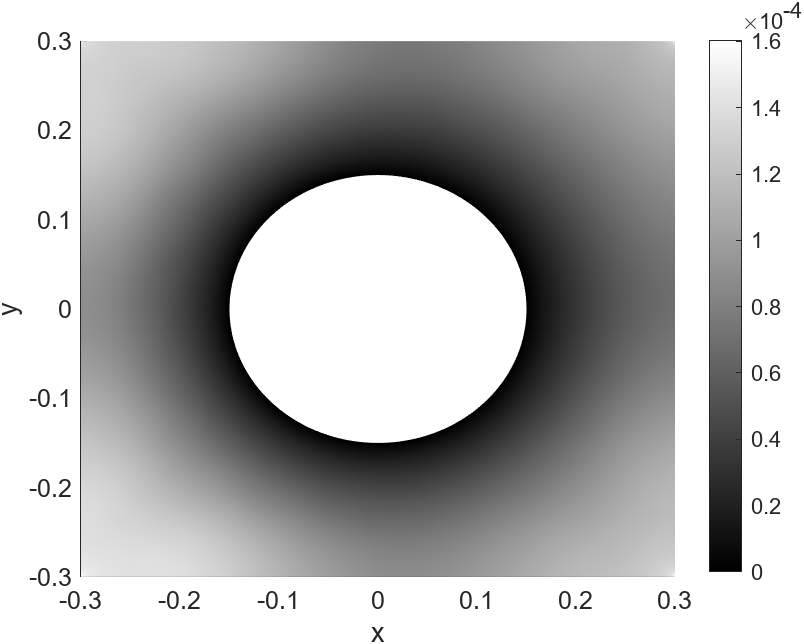}
			\caption{|Difference|  }
		\end{subfigure}
    \caption{Examples of iterates on shape 1, taking $k=1.93, \alpha=k^2$ (must converge according to Theorem \ref{lemma: second version conv}). The solution to \eqref{eq: origin helmholtz} is shown for comparison. (Real part of solution in all cases.)  
     }
    \label{fig:3-1}
\end{figure}

Table \ref{tab:my-table0} explores the sufficient condition for convergence in Theorem \ref{lemma: second version conv}, with different geometries and choices of $k$.   
We find that the iterate sums do converge (as expected) to the true solution whenever $k\leq k_*^A$, for the $k_*^A$ corresponding to the given shape. Moreover, while the iterative scheme also works for some $k$ past $k_*^A$, it eventually breaks down in all cases at some $k\gtrsim k_*^A$---
as expected from Section \ref{SS:Algebra}. Therefore, the upper bound for sufficient convergence ($k_*^A$) must be regarded as a lower bound to the upper bound for numerical convergence, as discussed in Section \ref{SS:Algebra}. The ``error'' column on Table \ref{tab:my-table0} lists the sup-norm of the difference between the real part of the exact solution and $\sum_{n=0}^{N}v_n$, for $N=30$. We can observe that as $k$ grows towards and past $k_*^A$, the convergence rate slows down.

\begin{table}[h!tb]
\centering
\begin{tabular}{|c|c|c|c|c|c|c|}
\hline
Shape              & $\sup_{{\bf y} \in \Omega}\bE[\tau^{\bf y}] $            &$\sup_{{\bf y} \in \Omega} \bE[\xi^{\bf y} ]$    &$k_*^A$                & Tested $k$    & Error    & Result \\ 
\hline
\multirow{4}{*}{1} & \multirow{4}{*}{0.03} & \multirow{4}{*}{0.46} & \multirow{4}{*}{1.93} & 1.50 & 9.13E-05 & converges \\ \cline{5-7} 
                   &                       &                       &                       & 1.93 & 1.72E-04 & converges \\ \cline{5-7} 
                   &                       &                       &                       & 1.94 & 1.76E-04 & converges \\ \cline{5-7} 
                   &                       &                       &                       & 2.90 & 1.5748   & diverges  \\ \hline
\multirow{4}{*}{2} & \multirow{4}{*}{0.15} & \multirow{4}{*}{0.91} & \multirow{4}{*}{0.95} & 0.70 & 2.40E-06 & converges \\ \cline{5-7} 
                   &                       &                       &                       & 0.95 & 9.20E-03 & converges \\ \cline{5-7} 
                   &                       &                       &                       & 0.96 & 1.19E-02 & converges \\ \cline{5-7} 
                   &                       &                       &                       & 1.12 & 0.12     & diverges  \\ \hline
\multirow{4}{*}{3} & \multirow{4}{*}{0.20}  & \multirow{4}{*}{1.24} & \multirow{4}{*}{0.72} & 0.50 & 1.77E-05 & converges \\ \cline{5-7} 
                   &                       &                       &                       & 0.72 & 7.39E-05 & converges \\ \cline{5-7} 
                   &                       &                       &                       & 0.73 & 9.07E-05 & converges \\ \cline{5-7} 
                   &                       &                       &                       & 1.07 & 4.98     & diverges  \\ 
                  \hline
\end{tabular}
\caption{$\alpha=k^2$, $N=30$, and the FE mesh size is given by {\tt PDETool}'s parameter $H_{max}=0.01$.}
\label{tab:my-table0}
\end{table}

Figure \ref{fig: lines2} shows that for $k=k_*^A$, the sup norm of the iterations decays monotonously as expected from \eqref{eq: proof supremum norm decay} from  Theorem \ref{lemma: second version conv}. 

\begin{figure}[h!bt]
    \centering
 			\includegraphics[scale=0.30]{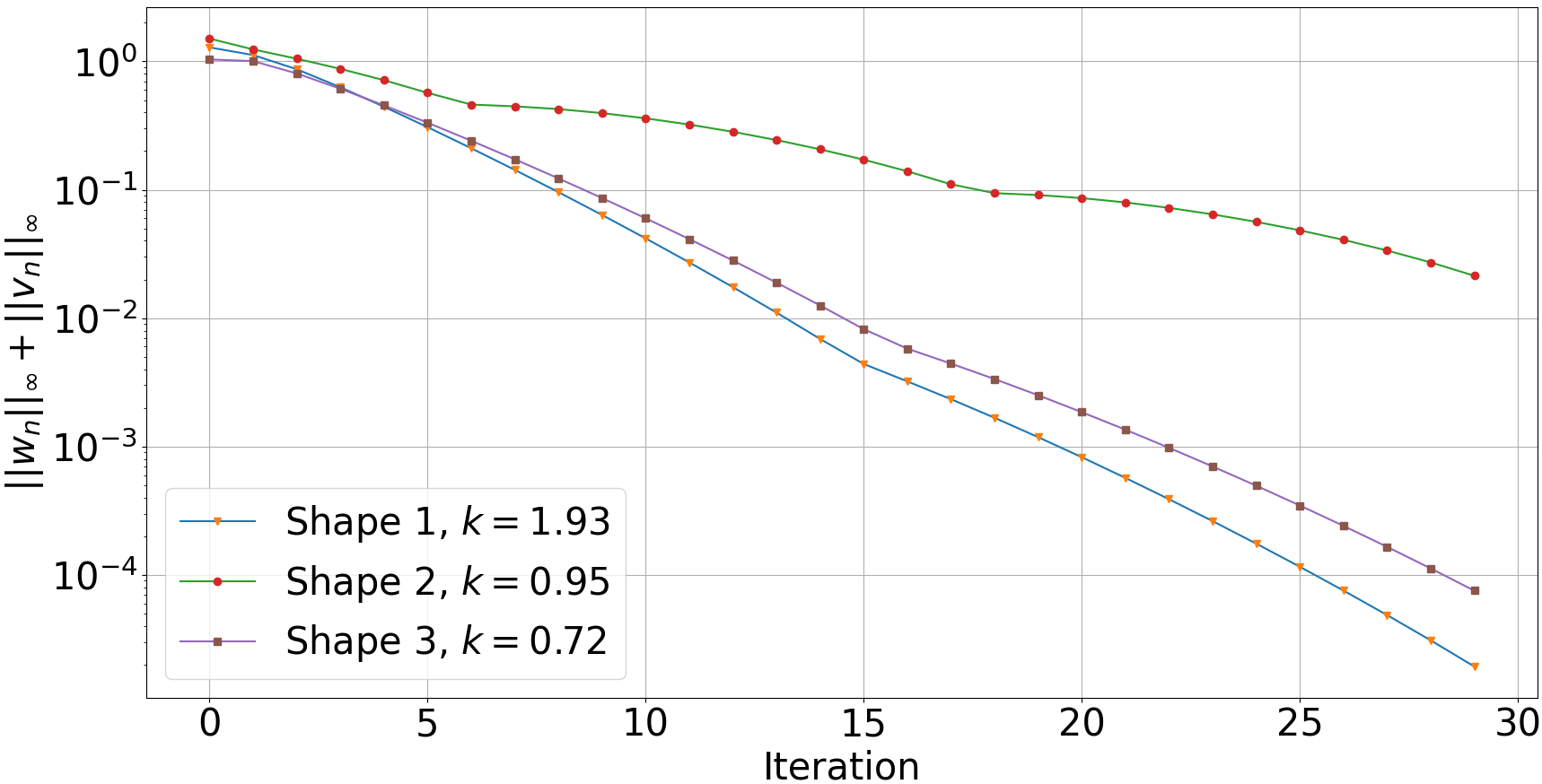}
    \caption{Convergence of $||v_n||_\infty+||w_n||_\infty$ using $k=k_*^A$ for the three shapes. Again $H_{max}=0.01$.
     }
    \label{fig: lines2}
\end{figure}

\section{Discussion on results' scope and limitations}\label{S:Discussion}
\paragraph{Positive results.} The analysis in the previous section proves that it is possible to cast the complex-valued Helmholtz BVPs related to wave propagation into an iteration of real-valued Poisson BVPs. In practice, the iteration would be truncated after $2m$ iterates---according, for instance, to a control on $\|v_m-v_{m-1}\|_{\infty}$. \textit{A priori}, it is not obvious how substituting one single linear problem (the original Helmholtz BVP) by a sequence of $2m$ linear problems is advantageous. The point of doing so---and the motivation which prompted this study---is solving the BVPs in the framework of Probabilistic Domain Decomposition (PDD).

\begin{figure}[hbt]
	\centering
	\includegraphics[scale=0.5]{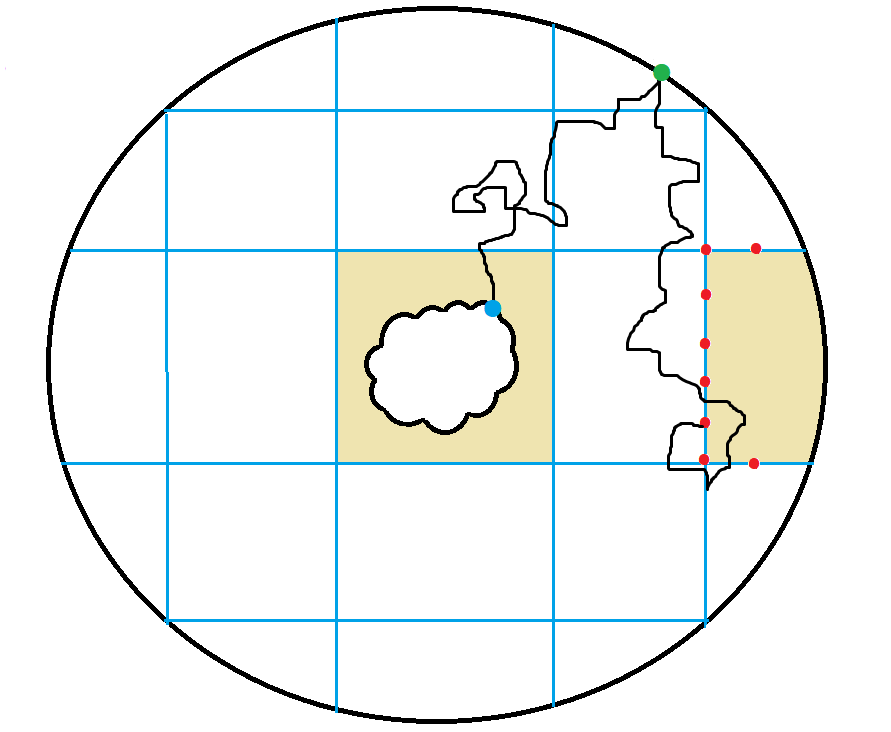}
	\caption{The annular computational domain $\Omega_c$ in the wave scattering problem, discretised into nonoverlapping subdomains. With PDD, the pointwise solution of the Poisson iterate is calculated by averaging the Feynman-Kac scores of an ensemble of stochastic diffusions starting at that point. One such trajectory is depicted. It starts at an interfacial node, diffuses throughout $\Omega_c$ according to the Brownian motion, may bounce off the computational boundary (where BCs of reflecting type hold), and eventually hits the Dirichlet BC on the scatterer's surface. After all interfacial values have been calculated, each subdomain (two of them are highlighted) has a well-posed Poisson BVP.}
	\label{F:PDD}
\end{figure}

Let us illustrate the notion of PDD with an example. Consider again the scattering of a plane wave. The annular computational domain $\Omega_c$ overlaid with a grid on nonoverlapping subdomains is sketched in Figure \ref{F:PDD}. Note that this grid and the underlying discretisation of each subdomain into nodes (by finite differences, meshless, or finite elements) is the same for PDD or for a (nonoverlapping) state of the art domain decomposition algorithm. 

In order to solve this problem with PDD, the $2m$ Poisson BVPs must be successively solved. Poisson BVPs are straightforward to tackle with PDD (see, e.g.~\cite{Acebron_PDD}). The procedure is:
\begin{itemize}
	\item First, the pointwise solution of the Poisson iterate BVP is computed on the {\em interfacial nodes only}, by averaging the Feynman-Kac scores of many independent diffusions inside $\Omega_c$.
	\item After that, the pointwise interfacial solutions are interpolated, and serve as a Dirichlet BC for each subdomain---which are solved independently by finite differences, elements, etc.
\end{itemize}
Note that the bulk of computations involved in both stages can be carried out concurrently, thus making near-optimal use of the parallel computer (i.e. with almost perfect strong scalability \cite{smith2017HybridPDEBranching}). This is, indeed, the purpose of PDD, and it is particularly relevant to wave scattering problems---where the scalability of state of the art domain decomposition may collapse due to the  huge number of DoFs and subdomains involved. In addition to enhanced scalability, the fact that most subdomain-restricted BVPs are Poisson equations with Dirichlet-only BCs allows for reliable, guaranteed convergent, preconditioners for definite positive algebraic systems\footnote{The discretisation matrix of a Poisson iterate restricted to a PDD subdomain with reflecting BCs (on the computational boundary) is not positive definite, as pointed out in Section \ref{SS:Algebra}.}, thus speeding up the convergence of Krylov iterations. (Incidentally, all this remains true if the sequence of Poisson BVPs \eqref{eq:naive_iter}---put forward in this paper---were to be tackled by standard domain decomposition methods, rather than PDD.) In fact, with PDD, the subdomains can conceivably be chosen small enough (in terms of DoFs) that the respective algebraic linear systems could be inverted by means of direct solvers, instead of iteratively---essentially circumventing one of the most notorious numerical issues of Helmholtz. Also, small enough to allow for pollution-free spectral methods on each of them---without worrying about the bloated or full algebraic bandwidth that prevents their use with the state of the art domain decomposition.

True, the two-stage PDD approach above must be repeated for every Poisson iterate, instead of just once as would be the case if the complex Helmholtz equation admitted a probabilistic representation. A ``multigrid'' variance reduction technique, which recycles information from each iterate and has proved very successful \cite{PDDMultigrid}, should be incorporated to accelerate the Poisson iterates.

\paragraph{Negative results.}
The formulae derived in Section 2 also indicate that recasting 
\eqref{eq:H0} into a sequence of Poisson equations 
is seemingly only possible for low wavenumbers / frequencies. 

The assumptions $(\alpha+p)\bE[\tau^{\bf y}]<1$ for the cavity (Theorem \ref{lemma: second version conv int}) and $(\alpha\bE[\tau^{\bf y} ]+ k \bE[\xi^{\bf y} ])<1$ for the annulus (Theorem \ref{lemma: second version conv}), along with the condition $\alpha\geq k^2$, give upper wavenumbers $k_*$ below which the Poisson iteration is guaranteed to converge to the solution of its respective Helmholtz BVP. Setting $\alpha=k_*^2=k^2$, recall that they are (note that all expectations are positive):
\begin{align*}
k_*^C&= \sqrt{\frac{1}{\bE[\tau^{\bf y}]}-p}~, 
& \textrm{ for the Helmholtz BVP in the cavity, and}
\\
k_*^A&= \min\limits_{{\bf y}\in\Omega }\frac{-\bE[\xi^{\bf y} ]+\sqrt{\bE^2[\xi^{\bf y}] +4\bE[\tau^{\bf y} ]}}{2\bE[\tau^{\bf y} ]},
& \textrm{ for the Helmholtz BVP on the annular domain}.
\end{align*} 
Let $E_{\Omega}({\bf x})={\mathbb E}[\tau^{\Omega}({\bf x})]$. The expected first-exit time of a Brownian motion from $\Omega$ obeys the PDE $\Delta E_{\Omega}({\bf x})+2=0$,
along with homogeneous BCs of the same type (Dirichlet, Neumann or Robin) as the associated BVP has on its boundary $\partial\Omega$ \cite{eikonal}. It can be readily shown (by inspection) 
that $E_{\Omega}({\bf x})$
scales with the area of the domain, i.e. if $\Lambda>0$ and $\Omega'=\{\Lambda{\bf y}\,|\,{\bf y}\in\Omega\}$, then $E_{\Omega'}(\Lambda {\bf x})= \Lambda^2E_{\Omega}({\bf x})$. Therefore, the upper limit $k_*^A$ decreases roughly with the area of the computational domain as well. This means that PDD cannot be used with {\em large} computational domains at {\em high} frequencies---which are two of the features of the scattering problem causing the most difficulties, in the first place.

Regarding the cavity problem, the domain is fixed for all frequencies (since no computational boundary is required). However, the upper bound $k_*^C$ still behaves qualitatively bad, since it decreases with the amount of damping, $p$. (While artificial damping is typically increased at high frequencies, in order to counter the resonance issue outlined in the Introduction.)    

Another typical scenario in wave propagation is a waveguide---a conduit designed to allow only a given frequency through it. From the stochastic point of view, the analysis (and the discouraging conclusion) are similar to those of the scattering problem---see Appendix \ref{S:Waveguide} for an example.

In principle, the bounds $k_*^A$ and $k_*^C$ indicate just {\em sufficient} conditions for convergence, thus not {\em necessarily} precluding convergence at higher frequencies. 
The experiments on Table \ref{tab:my-table0} (as well as others not reported) strongly suggest that there is a maximum wavenumber $k_{max}$ in every case beyond which the Poisson iterations \eqref{eq: def of pde v2 int} and \eqref{eq: def of pde v2} do diverge. 
Moreover, $k_*^A$ and $k_*^C$ are sharpish lower bounds to $k_{max}$ in either case---for practical purposes, $k\leq k_*^A$ or $k\leq k_*^C$ can be regarded almost as a necessary condition for convergence. This observation is backed up by the semiheuristic analysis carried out in Section \ref{SS:Algebra} for finite difference discretisations, which leverages completely independent arguments from those used to derive $k_*^A$ and $k_*^C$. 

Therefore, the iterations proposed in this paper (and  PDD) fail to fix any of the numerical challenges posed by Helmholtz equations in wave propagation. While the former could certainly be implemented in tame scenarios (small domains and low frequencies), they do not seem to offer any particular advantage over standard methods there---but a significantly more complex formulation.   

Summing up, despite its unusual and appealing features, the final verdict on the approach studied in this paper must be negative. The ultimate reason is the harmful effect of frequency and domain size on convergence through the expected first-exit times. While we have not been able to fix this, we finish this paper by briefly commenting on a few plausible ideas for further work.

First-exit times might be shortened (hence lifting $k_*$) by adding artificial drifts (i.e. first-derivative terms) to \eqref{eq:H0}, which push trajectories towards the Dirichlet BCs. In the Feynman-Kac setting, the connection between that modified problem and the Helmholtz equation would be established via Girsanov's theorem \cite{Milstein&Tretyakov04}. However, the convergence analysis in Sections \ref{SS:Cavity} and \ref{SS:Annular} becomes quite involved.

We have also tried, unsuccessfully, to improve on the qualitative behaviour of $k_*^A$ by designing a different iteration. We give an example of a sensible attempt in Appendix \ref{S:Fail}.

Perhaps the more promising idea might be replacing the fixed parameter $\alpha$ in the iterations \eqref{eq:naive_iter} by a function $\alpha({\bf x})$, such that $k^2-\alpha({\bf x})$ may be positive in a region of $\Omega_c$ (but not everywhere). This must be done in such a way that $\EE[\tau]$ and the variance are finite. We are not aware of any theoretical fact banishing this possibility nor are we of counterexamples in the literature.

\paragraph{Significance.}  
Our detailed analysis sheds a bittersweet conclusion: the notion we propose is feasible, but only at low frequencies. This need not be the last word on the approach, as the Feynman-Kac theory for elliptic BVPs with partially positive reaction coefficients is essentially non-existent. Nonetheless, and to the best of our knowledge, it does hold insofar as the currently available theory is concerned.

Given the originality of our approach, so different from the state of the art, we believe that our results are worth reporting to the broader community of applied mathematicians and engineers, where they could spur new---even out of the box---ideas.

\appendix
\section{Failed alternative iteration for annular domain}\label{S:Fail}

In this section, we consider the following alternative iterated BVP system for \eqref{eq: origin helmholtz}:
\begin{align}
\nonumber
    {\bf x}\in  \Omega,
    \begin{cases}
        \Delta v_0-(\alpha-k^2) v_0  =0,
        \\
        \Delta v_1 -(\alpha-k^2) v_1  =-\alpha v_0 ,
        \\
        \vdots
        \\
        \Delta v_n -(\alpha-k^2) v_n  =-\alpha v_{n-1},
        \\
        \vdots
    \end{cases}
    \quad
    \begin{cases}
        \Delta w_0-(\alpha-k^2) w_0  =0,
        \\
        \Delta w_1 -(\alpha-k^2) w_1  =-\alpha w_0 ,
        \\
        \vdots
        \\
        \Delta w_n -(\alpha-k^2) w_n  =-\alpha w_{n-1},
                \\
        \vdots
    \end{cases}
\end{align}
with the BCs 
\begin{align}
\label{eq: bc1 of pde v1}
  {\bf x}=(x_1,x_2)\in   \partial \Omega_1,~
    \begin{cases}
          v_0  =-\cos (kx_1),
        \\
         v_n =0, ~n\geq 1,
    \end{cases}
    \qquad \quad 
    \begin{cases}
          w_0  =\sin (kx_1),
        \\
         w_n =0, ~n\geq 1,
    \end{cases}
\end{align}
\begin{align}
\label{eq: bc2 of pde v1}
  {\bf x}\in   \partial  \Omega_2,
    \begin{cases}
          v_0  =0,
        \\
         v_n =-\frac{1}{k} \frac{\partial w_{n-1}}{\partial \bn}, ~n\geq 1,
    \end{cases}
    \qquad \quad 
    \begin{cases}
         \frac{\partial w_0}{\partial \bn}=0,
        \\
        \frac{\partial w_n}{\partial \bn}=k v_{n-1},  ~n\geq 1.
    \end{cases}
\end{align}
The difference between this version and the version in \eqref{eq: def of pde v2}-\eqref{eq: bc2 of pde v2} lies in the boundary condition for $v_n$, for which Dirichlet-only BCs are enforced on both $\partial\Omega_1$ and $\partial\Omega_2$ in \eqref{eq: bc1 of pde v1}-\eqref{eq: bc2 of pde v1}. The point is that, since the corresponding Feynman-Kac diffusion is now absorbed rather than reflected on $\partial\Omega_2$, its expected first exit time is shorter, and this might raise the upper end of the range of workable $\alpha$.

Let us introduce the two-dimensional standard Brownian motion $({\bf W}_{1,t})_{t\geq 0}$, the $\bR^2$ process $({\bf X}_t)_{t\geq 0 }$ absorbed on $\partial\Omega=\partial\Omega_1\cup\partial\Omega_2$, equipped with the standard probability space, such that
\begin{align}
   \nonumber
   \dd {\bf X}_t=\1_{\{{\bf X}_t\in\Omega\}}\sqrt{2}\, \dd {\bf W}_{1,t}, \qquad {\bf X}_0={\bf x} \in \Omega,
\end{align}
and its associated first exit time
\begin{align}
    \nonumber
	\tau_1^{\bf x}= \arg\min \{t>0\,|\,{\bf X}_t\in\partial\Omega\}.
\end{align}
Then, for $n\geq 2$, for all ${\bf x}\in \Omega$, we have 

    \begin{align}
     v_n({\bf x})&=\bE \Big[  \int_0^{\tau_1^{\bf x} } \alpha v_{n-1} ( {\bf X }_s^x)  e^{-(\alpha-k^2)s} \dd s
    \Big]
    +
    \bE \Big[  v_{n-2} ( {\bf X }_s^x)  e^{-(\alpha-k^2)\tau_1^{\bf x}} \1_{\{ X_{\tau_1^{\bf x}}\in \partial \Omega_2 \}}
    \Big]
    \nonumber\\
    &\leq ||v_{n-1}||_\infty\bE \Big[  \int_0^{\tau_1^{\bf x} } \alpha   e^{-(\alpha-k^2)s} \dd s
    \Big]
    +||v_{n-2}||_\infty
    \bE \Big[   e^{-(\alpha-k^2)\tau_1^{\bf x}} \1_{\{ X_{\tau_1^{\bf x}}\in \partial \Omega_2 \}}
    \Big]
    \nonumber\\
     &\leq \alpha ||v_{n-1}||_\infty\bE \Big[     \frac{1-e^{-(\alpha-k^2)\tau_1^{\bf x} }}{ \alpha-k^2}
    \Big]
    +||v_{n-2}||_\infty
    \bE \Big[   e^{-(\alpha-k^2)\tau_1^{\bf x}} \1_{\{ X_{\tau_1^{\bf x}}\in \partial \Omega_2\} }
    \Big]
    \nonumber
    \\\
     &\leq \alpha ||v_{n-1}||_\infty\bE \big[      \tau_1^{\bf x}  
    \big]
    +||v_{n-2}||_\infty
    \bP( X_{\tau_1^{\bf x}}\in \partial \Omega_2   ).
    \nonumber
\end{align}
In order for
$v_n$ to satisfy  $\lim_{n\rightarrow\infty} ||v_n||_\infty=0$, one would need   a condition that $\alpha \bE[\tau_1^{\bf x} ]+\bP( X_{\tau_1^{\bf x}}\in \partial \Omega_2 )<1 ~ \Rightarrow ~\alpha \bE[\tau_1^{\bf x} ] < \bP( X_{\tau_1^{\bf x}}\in \partial \Omega_1 )$, using that $\bP( X_{\tau_1^{\bf x}}\in \partial \Omega_1 )+ \bP( X_{\tau_1^{\bf x}}\in \partial \Omega_2 )=1$ . This condition restricts the feasible range of $\alpha$ to a narrower extent compared to the range established in Theorem \ref{lemma: second version conv}. This is primarily due to the relatively small probability $\bP( X_{\tau_1^{\bf x}}\in \partial \Omega_1 )$ associated with the annular geometry, where by design, the perimeter of $\partial\Omega_1$ is much shorter than that of $\partial\Omega_2$.

\section{Waveguide}\label{S:Waveguide}

The waveguide problem we consider is the rectangle $[0,1]\times[0,L_{wid}]$ with boundary $ \partial \Omega_{up}\cup \partial\Omega_{down}\cup\partial \Omega_{left} \cup\partial\Omega_{right}  $, corresponding to the four rectangle sides,  satisfying the BVP: 

\begin{align}
    \nonumber
        \begin{cases}
        \Delta u+k^2 u =0,
        &{\bf x}=(x_1,x_2)\in  \Omega,
        \\
        u=0, &{\bf x}\in \partial \Omega_{up} \cup \partial \Omega_{down} ,
        \\
        \frac{\partial u}{\partial \bn}+ik u= 2i \beta_m \sin(\alpha_m x_2) , &{\bf x}\in \partial \Omega_{left}, 
        \\
        \frac{\partial u}{\partial \bn}-ik u=0 , &{\bf x}\in \partial \Omega_{right },
    \end{cases}
\end{align}
where $\alpha_m=m \pi/L_{wid},~\beta_m=\sqrt{k^2 -\alpha_m^2  }$ and $m\in \bN$ (see \cite{waveguide} for details). Since the waveguide is elongated in the horizontal direction, $L_{wid}<1$. Even if the geometry of the waveguide is not annular, the problem is mathematically similar to the problem in Section \ref{SS:Annular}. Therefore, we propose the same iteration of PDEs \eqref{eq: def of pde v2}, but replace the BCs \eqref{eq: bc1 of pde v2} and \eqref{eq: bc2 of pde v2} by

\begin{align}
\label{eq: bc1 of pde v2 - waveguide}
   &{\bf x}\in  \partial \Omega_{up} \cup \partial\Omega_{down},
    \quad 
          v_n =w_n = 0 ,  ~n\geq 0,
\\
  &{\bf x}\in   \partial \Omega_{left},
    \begin{cases}
         \frac{\partial v_0}{\partial \bn}=0,
        \\
        \frac{\partial v_n}{\partial \bn}=k w_{n-1},  ~n\geq 1,
    \end{cases}
    \qquad \quad 
    \begin{cases}
         \frac{\partial w_0}{\partial \bn}=2 \beta_m \sin(\alpha_m x_2),
        \\
        \frac{\partial w_n}{\partial \bn}=-k v_{n-1},  ~n\geq 1,
    \end{cases}
    \\
  &{\bf x}\in   \partial\Omega_{right} ,
    \begin{cases}
         \frac{\partial v_0}{\partial \bn}=0,
        \\
        \frac{\partial v_n}{\partial \bn}=-k w_{n-1},  ~n\geq 1,
    \end{cases}
    \qquad \quad 
    \begin{cases}
         \frac{\partial w_0}{\partial \bn}=0,
        \\
        \frac{\partial w_n}{\partial \bn}=k v_{n-1},  ~n\geq 1.
    \end{cases}
\end{align}
Replicating the proof of Theorem \ref{lemma: second version conv},  we shall end up with a similar constraint that the iterations converge to the true solution if $(\alpha\bE[\tau_3^{\bf y} ]+ k \bE[\xi^{\bf y} ])   <1$, where $\tau_3^{\bf y}=\inf\{t>0\,|\,{\bf Y}_t^{\bf y}\in \partial \Omega_{up} \cup \partial\Omega_{down} \}$ is a stopping time similar to $\tau^{\bf y}$ defined \ref{lemma: second version conv}. For $m=1$, an experiment shows that we cannot get convergence for any $k$. Now, we prove that this condition is not attainable for $m\geq 2$, either. It is known that 
the expected exit time of  a two-dimensional Brownian motion from  the center of a ball of radius $R$ is $R^2/2$ \cite{alma99964413502466}. By symmetry, this is also the largest expected exit time from the ball. Then, the largest expected exit time of the scaled Brownian motion $({\bf Y}^{\bf y}_t)_{t\geq 0}$, defined in \eqref{eq:def_Y_t}, from the circle of radius $R=L/2$ concentric with---and inscribed in---the rectangle $\Omega$ is $L^2/16$. Consequently, $\sup_{y\in \Omega}\bE[\tau_3^{\bf y} ]> L_{wid}^2/16$. 
Additionally, by considering the iterated PDE, 
it follows that $k^2>\alpha_m^2\geq 4\pi^2 /L_{wid}^2$.
Then, $ \sup_{y\in \Omega}\alpha\bE[\tau_3^{\bf y} ] \geq k^2 \sup_{y\in \Omega}\bE[\tau_3^{\bf y} ] \geq 4\pi^2/16>1$, which violates the constraints. In sum, the iterative scheme cannot be guaranteed to converge for any value of $k$.


%



%
%

\bibliographystyle{abbrv}

\end{document}